\RequirePackage{fix-cm}

\documentclass[smallextended,numbook,runningheads]{svjour3}     
\smartqed  
\usepackage{graphicx}
\usepackage{amsmath}
\usepackage{epstopdf} 
\usepackage{mathptmx}      
\usepackage[T1]{fontenc}
\usepackage[utf8]{inputenc}
\usepackage{newtxtext,newtxmath}
\usepackage{hyperref}
\hypersetup{breaklinks,colorlinks,linkcolor=blue,citecolor=blue}
\usepackage{cleveref}
\usepackage{tikz}
\usepackage{pgfplots}
\pgfplotsset{compat=1.16}
\usepackage{comment}
\usepackage{booktabs}
\usepackage{array}

\usepackage{subcaption}
\usepackage{tikz}
\usepackage[nottoc]{tocbibind}
\usepackage{bm}
\usetikzlibrary{matrix, calc, arrows}

\usepackage{placeins}
\makeatletter
\renewcommand*\env@matrix[1][\arraystretch]{%
  \edef\arraystretch{#1}%
  \hskip -\arraycolsep
  \let\@ifnextchar\new@ifnextchar
  \array{*\c@MaxMatrixCols c}}
\makeatother

\newcommand{\WS}{\mathcal W}
\newcommand{\meq}{M_\text{eqn}}

\newcommand{\mdeg}{M_\text{deg}}
\newcommand{\Tm}{\mathcal T}
\newcommand{\half}{\frac{1}{2}}
\newtheorem{prop}{Proposition}
\newcommand{\reals}{\mathbb R}
\newcommand\figref{Figure~\ref}
\newcommand\tabref{Table~\ref}

\renewcommand{\vec}[1]{\ensuremath{\bm{\mathrm{#1}}}}
\newcommand{\mat}[1]{\ensuremath{\bm{\mathrm{#1}}}}

\newcommand{\diff}{\mathop{}\!\mathrm{d}}
\journalname{}
\begin{document}
\title{Semi-implicit Hybrid Discrete $\left(\text{H}^T_N\right)$ Approximation of Thermal Radiative Transfer
}

\titlerunning{Semi-implicit Hybrid Discrete Approximation of Thermal Radiative Transfer}

\author{Ryan G. McClarren \and James A. Rossmanith \and Minwoo Shin\footnote{Corresponding Author}}

\authorrunning{R.G. McClarren, J.A. Rossmanith, and M. Shin} 

\institute{   Ryan G. McClarren\at
              University of Notre Dame \\
              Department of Aerospace and Mechanical Engineering \\
              Notre Dame, IN 46556, USA\\
              \email{\href{mailto:rmcclarr@nd.edu}{rmcclarr@nd.edu}}
              \and 
              James A. Rossmanith\at
              Iowa State University \\
              Department of Mathematics \\
              Ames, IA 50011-2104, USA\\
              \email{\href{mailto:rossmani@iastate.edu}{rossmani@iastate.edu}}
              \and 
              Minwoo Shin\at
              University of Notre Dame \\
              Department of Aerospace and Mechanical Engineering \\
              Notre Dame, IN 46556, USA \\
              \email{\href{mailto:mshin@nd.edu}{mshin@nd.edu}}
}

\date{Received: date / Accepted: date}

\maketitle

\begin{abstract}
The thermal radiative transfer (TRT) equations form an integro-differential system that describes the propagation and collisional interactions of photons. Computing accurate and efficient numerical solutions TRT are challenging for several reasons, the first of which is that TRT is defined on a high-dimensional phase space that includes the independent variables of time, space, and velocity. In order to reduce the dimensionality of the phase space, classical approaches such as the P$_N$ (spherical harmonics) or the S$_N$ (discrete ordinates) ansatz are often used in the literature. In this work, we introduce a novel approach: the hybrid discrete (H$^T_N$) approximation to the radiative thermal transfer equations. This approach acquires desirable properties of both P$_N$ and S$_N$, and indeed reduces to each of these approximations in various limits: H$^1_N$ $\equiv$ P$_N$ and H$^T_0$ $\equiv$ S$_T$. We prove that H$^T_N$ results in a system of hyperbolic partial differential equations for all $T\ge 1$ and $N\ge 0$. Another challenge in solving the TRT system is the inherent stiffness due to the large timescale separation between propagation and collisions, especially in the diffusive  (i.e., highly collisional) regime. This stiffness challenge can be partially overcome via implicit time integration, although fully implicit methods may become computationally expensive due  to the strong nonlinearity and system size. On the other hand, explicit time-stepping schemes that are not also asymptotic-preserving in the highly collisional limit require resolving the mean-free path between collisions, making such schemes prohibitively expensive. In this work we develop a numerical method that is based on a nodal discontinuous Galerkin discretization in space, coupled with a semi-implicit discretization in time. In particular, we make use of a second order explicit Runge-Kutta scheme for the streaming term and an implicit Euler scheme for the material coupling term. Furthermore, in order to solve the material energy equation implicitly after each predictor and corrector step, we linearize the temperature term using a Taylor expansion; this avoids the need for an iterative procedure, and therefore improves efficiency. In order to reduce unphysical oscillation,  we apply a slope limiter after each time step. Finally, we conduct several numerical experiments to verify the accuracy, efficiency, and robustness of the H$^T_N$ ansatz and the numerical discretizations.
\end{abstract}


\section{Introduction}
The thermal radiative transfer (TRT) equations describe the interaction of matter and thermal radiation, and are used in a wide range of applications including remote sensing, glass manufacturing, and combustion \cite{siegel}. The TRT equations model the propagation, absorption, and emission
of photons, as well as the coupling of these dynamics to background media
via a material-energy equation. Accurate and efficient numerical solutions of the TRT equations are challenging for several reasons:
\begin{enumerate}

\item The TRT equations form an integro-differential system that is posed on a high-dimensional 
phase space:  the time dimension, up to three spatial dimensions, two velocity direction dimensions, and a frequency dimension;
\item Optically thick media introduces stiffness due to the large timescale separation between propagation and collisions;
\item The material-energy coupling equation introduces strong non-linearity.
\end{enumerate}

\subsection{\bf Discretization of the phase space}
Two broad classes of solution approaches have been used to tackle the TRT equation (1) stochastic methods and (2) deterministic methods; each approach offers their own advantages and disadvantages. The main workhorse of the stochastic approach is the Monte Carlo (MC) method, which is considered one of the most reliable methods in the radiation community \cite{fleck,mcclarren2009modified,wollaber2016four}.
Monte Carlo methods produce statistical noise due to under-sampling of the phase space; and therefore, to compensate for this under-sampling, many MC histories are required, which in turns makes the MC calculation expensive. As a result, high performance computing approaches are needed. 
{\color{black}In the realm of deterministic approaches for the angular dependence, the most popular methods are the discrete ordinates (S$_N$) method \cite{carlson,koch,lathrop3,lathrop2,lathrop1,pomraning1} and the spherical harmonics (P$_N$) method with various closures \cite{lewis,mcclarren1,olson1,olson,pomraning}.}
The advantages of these approaches are that, relative to Monte Carlo, they are less computationally expensive. On the other hand they each suffer from  their own specific disadvantages. 

The S$_N$ approximation scheme creates a system of equations along a discrete set of angular directions that are taken from a specific quadrature set \cite{carlson1,jarrell,lau,thurgood}, and angular integrals are calculated via the given quadrature set. The S$_N$ method has been the subject of intense research, and many large-scale efficient solution techniques have been developed, including approaches that have been shown to scale on leadership-class computers \cite{adams2020provably}.   The S$_N$ method's main drawback is the phenomenon known as ray effects \cite{lathrop3,lathrop2}, which arises due to the fact that particles move only along certain directions in the quadrature set. These effects are conspicuously observed in optically thin materials with localized sources or sharp material discontinuities. 
 
On the other hand, the P$_N$ method approximates the solution with a spherical harmonics expansion that by construction is rotationally invariant and converges in the $L^2$ sense \cite{hauck}. The main disadvantage of the P$_N$ approach is that it is plagued by Gibbs phenomena (i.e., unphysical oscillations) due to the fact that it is a truncated spectral method and can thereby produce the negative particle concentrations \cite{hauck,mcclarren2008solutions}, which is physically undesirable. Additionally, in the P$_N$ method, when $N$ is small, the wave speeds are reduced because the system eigenvalues can be far from unity. {\color{black}Closures have been developed to deal some of the shortcomings of the P$_N$ method, such as the M$_N$ method \cite{dubroca1999etude,hauck2011high} and other techniques \cite{laiu,zheng2016moment,hauck2010positive}.}

\subsection{\bf Stiffness}
Another challenging aspect in accurately and efficiently solving the thermal radiative transfer (TRT) equations is that the system is inherently {\it stiff}. This stiffness is especially pervasive
in the diffusive  regime (i.e., highly collisional), where there is a large timescale separation between propagation and collisions. The reduction of the kinetic TRT system to a lower dimensional system of partial differential equations (PDEs) via an ansatz such as discrete ordinate (S$_N$) or spherical harmonics (P$_N$) does not alleviate this problem. 

From the perspective of numerical methods,
the stiffness challenge can be partially overcome via implicit time integration, although fully implicit methods may become computationally expensive due  to the strong nonlinearity and system size. On the other hand, explicit time-stepping schemes that are not also asymptotic-preserving in the highly collisional limit require resolving the mean-free path between collisions, making such schemes prohibitively expensive.

An efficient alternative to both fully implicit and fully explicit schemes are semi-implicit schemes, where, at least roughly-speaking, the  transport is handled explicitly, while the material coupling term (i.e., collisions) are handled implicitly. For example, Klar et al. \cite{klar,klar1} developed and analyzed an operator splitting approach that was shown to be asymptotically-preserving while at the same time relatively efficient in all regimes. A different semi-implicit approach for TRT was subsequently developed by McClarren et al. \cite{mcclarren1}, in which a two-stage second-order Runge-Kutta (RK) scheme was used for the streaming term and the backward Euler scheme was used for the material coupling term. This scheme was also shown to be asymptotic-preserving in the diffusive limit.

\subsection{\bf Strong nonlinearity in the material coupling term}
In the absence of material coupling, the thermal radiative transfer equations and its various reductions (e.g., P$_N$ or S$_N$) represent a system of linear constant-coefficient partial differential equations. The presence of the material coupling term introduces strong nonlinearity. For explicit time-stepping methods this nonlinearity does not pose a direct challenge, although, as mentioned above, explicit methods will then suffer from small time-steps. For fully implicit time-stepping methods the nonlinearity results in large nonlinear algebraic equations that must be inverted in time-step, resulting in significant computational expense. For semi-implicit time-stepping approaches, a nonlinear material coupling term only introduces a local nonlinear algebraic equation (i.e., local on each element), although even this can add computational expense. Fortunately, McClarren, Evans, and Lowrie \cite{mcclarren1} showed that  the local nonlinear problems presented by a semi-implicit approach can be linearized via a simple Taylor series argument. It was shown that this linearization reduces the computational complexity and does not adversely affect overall accuracy of the method.

\subsection{\bf Scope of this work}
 
In order to address some of the shortcomings of classical deterministic methods, we introduce in this work the hybrid discrete (H$^T_N$) approximation, which hybridizes aspects of the P$_N$ and S$_N$ methods. The H$^T_N$ approximation is equivalent to spherical harmonics (P$_N$) approximation when the number of the discretized velocity space ($T$) is one and only depends on the order of spherical harmonics basis functions ($N$). Also, it is equivalent to the discrete ordinate (S$_N$) method for a certain quadrature set. By hybridizing P$_N$ and S$_N$, the H$^T_N$ approximation is able to acquire beneficial properties of both of these classical approaches. We prove that H$^T_N$ results in a system of hyperbolic partial differential equations for all $T\ge 1$ and $N\ge 0$.

Once we have shown how to reduce the TRT equations into their H$^T_N$ approximate form, we will discretize the resulting PDE system via the same semi-implicit time integration scheme introduced by McClarren et al. \cite{mcclarren1}. In particular, this approach uses a second order Runge-Kutta explicit time discretization schemes for the streaming term and a backward Euler scheme for the material energy term; this allows us to resolve the stiffness of the TRT systems and to preserve the asymptotic diffusion limit.
Furthermore, in order to solve the material energy equation implicitly after each predictor and corrector step, we linearize the temperature term using a Taylor expansion; this avoids the need for an iterative procedure, and therefore improves efficiency. In order to reduce unphysical oscillation,  we apply a slope limiter after each time step. 

In the present work we will consider only grey, i.e., frequency-averaged, radiation transport equations in slab geometry. Non-grey and multi-dimensional radiative transfer will be studied in future work.  For the grey TRT system in slab geometry we consider several standard test cases in order to validate the accuracy, efficiency, and robustness of scheme, as well as to highlight benefits over P$_N$ and S$_N$ solutions.
 
 The remainder of this paper is organized as follows. In  \cref{sec:hd1d} the novel hybrid discrete (H$^T_N$) approximations in slab geometry are derived ands discussed. In \cref{sec:dg} we develop a semi-implicit nodal discontinuous Galerkin finite element scheme for the resulting H$^T_N$ systems. In  \cref{sec:numerical_result} we provide numerical results of H$^T_N$ approximations on various benchmark problems to show its robustness and asymptotic preserving property. Conclusions are presented in  \cref{sec:conclusion}.


\section{Hybrid discrete approximation} 
\label{sec:hd1d}
In this section, we formulate H$^T_N$ approximations to the radiative transfer in one-dimension. 
We begin with the thermal radiative transfer in slab geometry, and then derive the 
hybrid discrete approximation. We prove that the H$^T_N$ system is always hyperbolic.
 
\subsection{\bf Frequency-independent grey thermal transfer}\label{sec:hd1d-1}
Consider the 1D scattering-free thermal radiative transfer equation (e.g., see \cite{mcclarren1}):
\begin{equation}\label{eq:radiative1d_a}
\frac{1}{c}\frac{\uppartial \widehat{I}}{\uppartial t} + \mu \frac{\uppartial\widehat{I}}{\uppartial z}  + \widehat{\sigma} \widehat{I}= \widehat{\sigma} \widehat{B}+\frac{\widehat{s}}{2}, 
\end{equation}
where $\widehat{I}(t,z,\mu,\nu): \reals_{\ge 0} \times \reals \times [-1,1] \times \reals \mapsto \reals$ is the specific intensity, $t\in\reals_{\ge 0}$ is time, $z\in\reals$ is the spatial variable, $\mu\in[-1,1]$ is the angular variable, i.e., $\mu=\cos{\varphi}$ with polar angle $\varphi
\in [0,\pi]$, $\nu\in\reals$ is the frequency of the photon,  $\widehat{s}(t,z,\nu):\reals_{\ge 0} \times \reals \times \reals \mapsto \reals$ is an external source term, $\widehat\sigma (z,\nu): \reals \times \reals \mapsto \reals_{\ge 0}$ is the opacity, $\widehat{B}(\nu,\theta): \reals \times \reals_{\ge 0} \mapsto \reals_{\ge 0}$ 
is the Planck function that satisfies:
\begin{equation}
\label{eqn:b_theta}
    B(\theta):=\int_{\nu}\widehat{B}(\nu,\theta) \diff\nu=\frac{ac\theta^4}{2},
\end{equation}
where  $\theta(t,z): \reals_{\ge 0} \times \reals  \mapsto \reals_{\ge 0}$ is the material temperature in keV, $c=3\times10^{10}$cm s$^{-1}$ is the speed of light, and $a=1.372\times10^{14}$ ergs cm$^{-3}$ keV$^{-4}$ is the radiation constant. The notation $\widehat{\cdot}$ is used to signify quantities that depend explicitly on the frequency $\nu$.
Equation \eqref{eq:radiative1d_a} couples to the material-energy equation:
\begin{equation}\label{eq:radiative1d_b}
\frac{\uppartial e}{\uppartial t}=\int_{\mu}\int_{\nu}{\widehat{\sigma}(z,\nu) \left(\widehat{I}(t,z,\mu,\nu)-\widehat{B}(\nu,\theta)\right)}\diff\nu\diff\mu,
\end{equation}
where $e(\rho,\theta):\reals_{\ge 0} \times \reals_{\ge 0} \mapsto \reals_{\ge 0}$
is the material energy per volume and $\rho(t,z):\reals_{\ge 0} \times \reals  \mapsto
\reals_{\ge 0}$ is the material density.

For the remainder of this paper we consider only the case where the opacity is frequency independent: $\widehat{\sigma}(z,\nu) \rightarrow \sigma(z)$. In this approximation, we define the
following frequency-integrated quantities (which now removes the $\widehat{\cdot}$ notation):
\begin{equation}
I(t,z,\mu) := \int_{\nu} \widehat{I}(t,z,\mu,\nu) \, \diff\nu, \quad
s(t,z) := \int_{\nu} \widehat{s}(t,z,\nu) \, \diff\nu,
\end{equation}
and equation \eqref{eqn:b_theta}. We also define
the angular moment of the radiation intensity $\widehat{I}$ as follows:
\begin{align}
E(t,z):=\frac{1}{c}\int_{\mu}\int_{\nu}\widehat{I}(t,z,\mu,\nu)\diff\nu\diff\mu
 = \frac{1}{c} \int_{\mu} I(t,z,\mu) \, \diff\mu.
\label{eq:moment1}
\end{align}
Furthermore, we will assume throughout this work that the material density, $\rho$,
 is constant, which results in the following:
\begin{equation}\label{eq:energy_temperature}
    \frac{\uppartial e}{\uppartial t}=\frac{\uppartial{e}}{\uppartial{\theta}}\frac{\uppartial{\theta}}{\uppartial{t}}+\frac{\uppartial{e}}{\uppartial{\rho}}\frac{\uppartial{\rho}}{\uppartial{t}}=C_{v}\frac{\uppartial{\theta}}{\uppartial{t}},
\end{equation}
where $C_{v} = 0.3 \times 10^{16} \text{erg}/\text{cm}^3/\text{keV}$ is the heat capacity at constant volume.

Under the assumption of a frequency independent opacity and a constant material density,
we can integrate \eqref{eq:radiative1d_a} over the frequency $\nu$ and arrive at the 
following grey transport equation and material-energy equation:
\begin{gather}\label{eq:frequencyfree_grey1d}
    \frac{1}{c}\frac{\uppartial I}{\uppartial t} + \mu \frac{\uppartial I}{\uppartial z}+\sigma I=\frac{1}{2}\sigma{ac{\theta}^4}+\frac{s}{2}, \\
\label{eq:material1d}
    C_v \frac{\uppartial \theta}{\uppartial t}=\sigma \left( \int_{-1}^{1} I(t,z,\mu) \, \diff\mu -a c {\theta}^4 \right).
\end{gather}
respectively. These equations are defined on $t>0$, $z \in \left(z_L,z_R\right)$, and $\mu \in [-1,1]$, and must be equipped with initial conditions at $t=0$ and appropriate boundary conditions at $z=z_L$ and $z=z_R$.

\subsection{\bf Formulation of the 1D H$^T_N$ approximation}\label{sec:hd1d-2}
Equations \eqref{eq:frequencyfree_grey1d}--\eqref{eq:material1d} represent an integro-differential equation for the intensity, $I(t,z,\mu)$, and the material temperature, $\theta(t,z)$. {\color{black}One class of techniques for approximating such systems is the dimension reduction via ansatz in $\mu$.} Standard techniques for reducing the dimensionality of these equations to a system defined only over $(t,z)$ include the P$_N$  (i.e., spherical harmonics) (e.g., see \cite{brunner1}) and S$_N$ (i.e., discrete ordinates) \cite{carlson} methods. Spherical harmonics (P$_N$) suffer from producing negative (and therefore unphysical) intensities, while discrete ordinates (S$_N$) suffer from ray effects. In this work we consider an alternative approach that was first developed in Shin \cite{shin}, namely the H$^T_N$ (i.e., hybrid discrete) approach. The main advantage of H$^T_N$ is that it allows us to  work with an approximation that combines desirable aspects of both P$_N$ and S$_N$. In particular H$^T_N$ reduces to each of these approximations in various limits: H$^1_N$ $\equiv$ P$_N$ and H$^T_0$ $\equiv$ S$_T$.

\begin{figure}[htbp]
\centering
\begin{tabular}{cc}
(a)\hspace{-2mm}\begin{tikzpicture}
  \node [anchor=east,magenta] at (0.075,2.175) {$z_L$};
  \node [anchor=west,magenta] at (3.95,2.175) {$z_R$};
  \draw[magenta] (0,2) circle (2pt);
  \draw[magenta] (4,2) circle (2pt);
  \draw[thick] (0, 0) rectangle (4, 4);
  \draw[->,black,thick](0, 2) -- (4.5, 2) node[anchor=north]{$z$};
  \draw[black,thick](-1, 2) -- (0, 2);
  \draw[->,black,thick](0, -0.5) -- (0, 4.5) node[anchor=east]{$\mu$};  
  \draw[-,blue,thick] (0, 0.5) -- (4, 0.5);
  \draw[-,blue,thick] (0, 1.0) -- (4, 1.0);
  \draw[-,blue,thick] (0, 1.5) -- (4, 1.5);
  \draw[-,blue,thick] (0, 2.5) -- (4, 2.5);
  \draw[-,blue,thick] (0, 3.0) -- (4, 3.0);
  \draw[-,blue,thick] (0, 3.5) -- (4, 3.5);
  \node at (0,0.25) [circle,fill,inner sep=1.5pt]{};
  \node at (0,0.75) [circle,fill,inner sep=1.5pt]{};
  \node at (0,1.25) [circle,fill,inner sep=1.5pt]{};
  \node at (0,2.75) [circle,fill,inner sep=1.5pt]{};  
  \node at (0,3.25) [circle,fill,inner sep=1.5pt]{};
  \node at (0,3.75) [circle,fill,inner sep=1.5pt]{};
  \draw[magenta] (0,0) circle (2pt);
  \draw[magenta] (0,4) circle (2pt);
  \node [anchor=east,magenta] at (0,0.00) {$-1$};
  \node [anchor=east] at (0,0.25) {$\mu_1$};
  \node [anchor=east] at (0,0.75) {$\mu_2$};
  \node [anchor=east] at (0,1.25) {$\mu_3$};
  \node [anchor=east] at (0,2.75) {$\mu_{T-2}$};
  \node [anchor=east] at (0,3.25) {$\mu_{T-1}$};
  \node [anchor=east] at (0,3.75) {$\mu_{T}$};
  \node [anchor=east,magenta] at (0,4.00) {$1$};
  \node at (-0.45,2.125+0.25) [circle,fill,inner sep=0.5pt]{};  
  \node at (-0.45,2+0.25) [circle,fill,inner sep=0.5pt]{};    
  \node at (-0.45,1.875+0.25) [circle,fill,inner sep=0.5pt]{};
  \node at (-0.45,2.125-0.25) [circle,fill,inner sep=0.5pt]{};  
  \node at (-0.45,2-0.25) [circle,fill,inner sep=0.5pt]{};    
  \node at (-0.45,1.875-0.25) [circle,fill,inner sep=0.5pt]{};
\end{tikzpicture} & 
(b)\hspace{-2mm}\begin{tikzpicture}
  \node [anchor=east,magenta] at (0.075,2.175) {$z_L$};
  \node [anchor=west,magenta] at (3.95,2.175) {$z_R$};
  \draw[magenta] (0,2) circle (2pt);
  \draw[magenta] (4,2) circle (2pt);
  \draw[thick] (0, 0) rectangle (4, 4);
  \draw[->,black,thick](0, 2) -- (4.5, 2) node[anchor=north]{$z$};
  \draw[black,thick](-1, 2) -- (0, 2);
  \draw[->,black,thick](0, -0.5) -- (0, 4.5) node[anchor=east]{$\mu$};  
  \draw[-,blue,thick] (0, 0.5) -- (4, 0.5);
  \draw[-,blue,thick] (0, 1.0) -- (4, 1.0);
  \draw[-,blue,thick] (0, 1.5) -- (4, 1.5);
  \draw[-,blue,thick] (0, 2.5) -- (4, 2.5);
  \draw[-,blue,thick] (0, 3.0) -- (4, 3.0);
  \draw[-,blue,thick] (0, 3.5) -- (4, 3.5);
  \node at (0,0.25) [circle,fill,inner sep=1.5pt]{};
  \node at (0,0.75) [circle,fill,inner sep=1.5pt]{};
  \node at (0,1.25) [circle,fill,inner sep=1.5pt]{};
  \node at (0,2.75) [circle,fill,inner sep=1.5pt]{};  
  \node at (0,3.25) [circle,fill,inner sep=1.5pt]{};
  \node at (0,3.75) [circle,fill,inner sep=1.5pt]{};
  \draw[magenta] (0,0) circle (2pt);
  \draw[magenta] (0,4) circle (2pt);
  \node [anchor=east,magenta] at (0,0.00) {$-1$};
  \node [anchor=east] at (0,0.25) {$\mu_1$};
  \node [anchor=east] at (0,0.75) {$\mu_2$};
  \node [anchor=east] at (0,1.25) {$\mu_3$};
  \node [anchor=east] at (0,2.75) {$\mu_{T-2}$};
  \node [anchor=east] at (0,3.25) {$\mu_{T-1}$};
  \node [anchor=east] at (0,3.75) {$\mu_{T}$};
  \node [anchor=east,magenta] at (0,4.00) {$1$};
  \node at (-0.45,2.125+0.25) [circle,fill,inner sep=0.5pt]{};  
  \node at (-0.45,2+0.25) [circle,fill,inner sep=0.5pt]{};    
  \node at (-0.45,1.875+0.25) [circle,fill,inner sep=0.5pt]{};
  \node at (-0.45,2.125-0.25) [circle,fill,inner sep=0.5pt]{};  
  \node at (-0.45,2-0.25) [circle,fill,inner sep=0.5pt]{};    
  \node at (-0.45,1.875-0.25) [circle,fill,inner sep=0.5pt]{};
  \draw[-,gray,thick] (0.75,0) -- (0.75,4);
  \draw[-,gray,thick] (1.5,0) -- (1.5,4);
  \draw[-,gray,thick] (2.50,0) -- (2.50,4);
  \draw[-,gray,thick] (3.25,0) -- (3.25,4);  
  \node [anchor=north] at (0.375,1.95) {$z_1$};
  \node [anchor=north] at (1.125,1.95) {$z_2$};
  \node at (0.375,2) [circle,fill,inner sep=1.5pt]{};
  \node at (1.125,2) [circle,fill,inner sep=1.5pt]{};
  \node at (2.875,2) [circle,fill,inner sep=1.5pt]{};
  \node at (3.625,2) [circle,fill,inner sep=1.5pt]{};
  \node [anchor=north] at (2.875,1.95) {$z_{M-1}$};
  \node [anchor=north] at (3.625,1.95) {$z_{M}$};  
  \node at (1.875,2.25) [circle,fill,inner sep=0.5pt]{};
  \node at (2.000,2.25) [circle,fill,inner sep=0.5pt]{};    
  \node at (2.125,2.25) [circle,fill,inner sep=0.5pt]{};
  \node at (1.875,1.75) [circle,fill,inner sep=0.5pt]{};
  \node at (2.000,1.75) [circle,fill,inner sep=0.5pt]{};    
  \node at (2.125,1.75) [circle,fill,inner sep=0.5pt]{};
\end{tikzpicture}
\end{tabular}
\caption{The H$^T_N$ approximation. Panel (a) shows the discrete velocity bands, each of which is
centered at  $\mu_j \in (-1,1)$ for $j=1,2,\ldots,m$. Panel (b) shows the discrete velocity bands with the physical $z$-mesh superimposed.\label{fig:mucells}}
\end{figure}
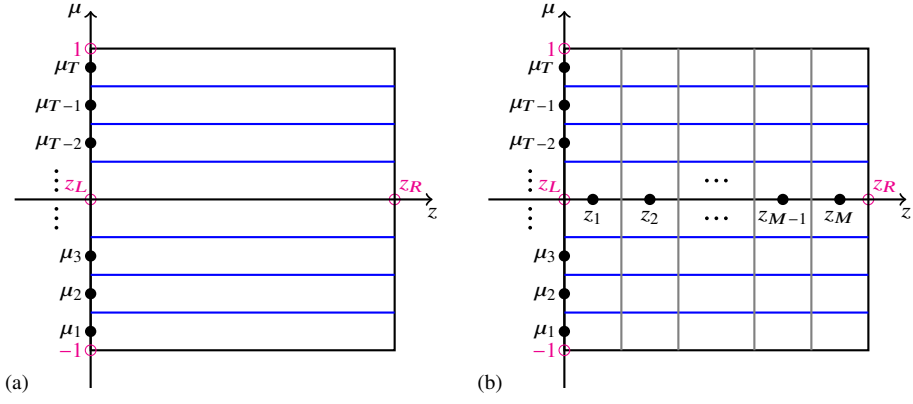

The H$^T_N$ approximation of Shin \cite{shin} begins by constructing a mesh in velocity space. In 1D this means a mesh in the angular variables $\mu \in [-1,1]$:
\begin{equation}
\label{eqn:mu_mesh}
 \bigcup_{j=1}^T \left[\mu_j - \frac{\Delta \mu}{2}, 
  \mu_j + \frac{\Delta \mu}{2} \right], \quad
  \mu_j = -1 + \left(j-\frac{1}{2} \right) \Delta \mu, \quad \text{and} \quad
  \Delta \mu = \frac{2}{T},
\end{equation}
 which we depict in
  \figref{fig:mucells}(a). With this velocity mesh, the H$^T_N$ approach
defines on the velocity band centered at $\mu_j$ the following band-localized intensity:
\begin{equation}\label{eq:hn_ansatz1d}
  {I}(t,z,\mu(\alpha))\bigg|_{\mu \in \left[\mu_j - \frac{\Delta \mu}{2}, 
  \mu_j + \frac{\Delta \mu}{2} \right]} \approx {I}^{(j)}(t,z,\alpha) := \sum_{k=0}^{N} \left(\frac{2k+1}{2} \right) u_k^{(j)}(t,z)p_{k}(\alpha),
\end{equation}
for $j=1,2,\ldots,T$, where $\alpha\in[-1,1]$ is a local variable in each velocity band centered at $\mu_j$ with thickness $\Delta \mu$:
\begin{equation}
 \mu(\alpha)=\mu_{j}+\alpha \left(\frac{\Delta\mu}{2}\right),
\end{equation}
and $p_k(\alpha)$ is the $k^{\text{th}}$ order Legendre polynomial:
\begin{equation}
\label{eqn:legendre_ortho}
\begin{split}
p_k(\alpha) = \left(\frac{2k-1}{k}\right) \alpha & p_{k-1}(\alpha) - \left(\frac{k-1}{k} \right)
p_{k-2}(\alpha), \quad
p_0(\alpha) = 1, \quad p_1(\alpha) = \alpha, \\ &\text{and} \quad
\int_{-1}^{1} p_k(\alpha) \, p_m(\alpha) \, \diff\alpha = \left(\frac{2}{2k+1}\right) \delta_{km}.
\end{split}
\end{equation}
The moments in \eqref{eq:hn_ansatz1d} are defined as
\begin{equation}\label{eq:rhoj1d}
  u^{(j)}_k(t,z)=\int_{-1}^{1}{{I}^{(j)}(t,z,\alpha) \, p_k(\alpha)} \, \diff\alpha,
\end{equation}
for $j=1,2,\cdots,T$. Note that in the H$^T_N$ approximation $T$ is the number of $\mu$-cells,
while $N$ is the number of basis functions used in ansatz \eqref{eq:hn_ansatz1d}.  

Using the H$^T_N$ approximation, on each interval $\mu \in \left[\mu_j - \frac{\Delta \mu}{2}, 
  \mu_j + \frac{\Delta \mu}{2} \right]$, equation \eqref{eq:frequencyfree_grey1d} can be written
  as follows (where we have also temporarily assumed that $s\equiv0$):
\begin{equation}\label{eq:discrete_frequencyfree_grey1d}
\frac{1}{c}\frac{\uppartial {I}^{(j)}}{\uppartial t}+\Big(\mu_j+\alpha\frac{\Delta{\mu}}{2}\Big) \frac{\uppartial {I}^{(j)}}{\uppartial z}+\sigma {I}^{(j)}=\frac{1}{2}{\sigma ac{\theta}^4},
\end{equation}
for $j=1,2,\cdots,T$ and $\alpha\in[-1,1]$. Plugging ansatz \eqref{eq:hn_ansatz1d} into \eqref{eq:discrete_frequencyfree_grey1d}, multiplying the equation by the Legendre polynomial $p_{\ell}$, integrating over $[-1,1]$, and using the orthogonality of
Legendre polynomials \eqref{eqn:legendre_ortho}, gives the following equations:
\begin{equation}
    \frac{1}{c}\frac{\uppartial {u}^{(j)}_{\ell}}{\uppartial t}+\mu_j \frac{\uppartial {u}^{(j)}_{\ell}}{\uppartial z}+\sum_{k=0}^{N}{\color{black}\frac{\Delta{\mu}}{2}} P_{\ell k}  \frac{\uppartial {u}^{(j)}_{k}}{\uppartial z} + \sigma {u}^{(j)}_{\ell}=\sigma ac{\theta}^4\delta_{\ell 0},
\end{equation}
for $\ell=0,1,\cdots,N$, where via the three-term recurrence relationship and orthogonality
from \eqref{eqn:legendre_ortho}:
\begin{equation}
\begin{split}
P_{\ell k} =& \left(\frac{2k+1}{2}\right) \int_{-1}^{1}{\alpha p_k(\alpha) p_{\ell}(\alpha)}\diff\alpha 
=  \left(\frac{k}{2}\right) \int_{-1}^{1}{p_{k-1}(\alpha) p_{\ell}(\alpha)}\diff\alpha \\ +& \left(\frac{k+1}{2}\right) \int_{-1}^{1}{p_{k+1}(\alpha) p_{\ell}(\alpha)}\diff\alpha
=
\begin{cases} 
      \frac{k}{2k-1}, & \text{if}\ \ell=k-1, \\
      \frac{k+1}{2k+3}, & \text{if}\ \ell=k+1, \\
      0, & \text{otherwise}.
   \end{cases}
   \end{split}
\end{equation}
We can write the resulting system of equations as follows:
\begin{equation}\label{eq:hn1d_2}
\frac{1}{c}\,\frac{\uppartial \vec{u}^{(j)}}{\uppartial t} + \mat{A}^{(j)} \, \frac{\uppartial \vec{u}^{(j)}}{\uppartial z} = \vec{q}\left(\vec{u}^{(j)},\theta\right),
\end{equation}
for $j=1,2,\cdots,T$, where
$\displaystyle \vec{u}^{(j)}=\begin{pmatrix}
         u_0^{(j)} & u_1^{(j)} & \dots & u_{N}^{(j)} \end{pmatrix}^{T} \in \reals^{N+1}$,
\begin{align}
\label{eq:matA1d_sub}
\mat{A}^{(j)} &= \mu_j \mat{I} + \frac{\Delta \mu}{2} \begin{pmatrix}[1.1]
    \ 0\ &\ 1&\ &\ &\ &\  \\
    \ \frac{1}{3}\ &\ 0&\ \frac{2}{3}&\ &\ &\  \\
    \ \ &\ \frac{2}{5}&\ 0&\ \ddots
    &\ &\ &\  \\
    \   &\  &\ \ddots &\ \ddots&\ \frac{N-1}{2N-3}
    &\   \\    
    \ \  &\  &\  &\  \frac{N-1}{2N-1}&\ 0&\ \frac{N}{2N-1} \\    
    \ \ &\ &\ &\ &\  \frac{N}{2N+1}&\ 0
\end{pmatrix}, \\
 \vec{q}\left( \vec{u}^{(j)}, \theta \right) &= -\sigma\vec{u}^{(j)}+
\sigma  ac{\theta}^4 \vec{e}_1,
\end{align}
where $\vec{e}_1=(1,0,\ldots,0)^T \in \reals^{(N+1)}$, 
$\mat{I},\mat{A}^{(j)} \in \reals^{(N+1)\times (N+1)}$, and
$\vec{q}\left( \vec{u}^{(j)},\theta \right): \reals^{(N+1)} \times \reals_{\ge 0} \mapsto \reals^{(N+1)}$.
Finally, we obtain the linear system:
\begin{equation}\label{eq:hyperbolic1d}
\frac{1}{c}\,\frac{\uppartial \vec{u}}{\uppartial t} +\mat{A}\ \frac{\uppartial \vec{u}}{\uppartial z}=\vec{Q}\left( \vec{u}, \theta \right),
\end{equation}
where
\begin{equation}\label{eq:matA1d}
\mat{A}=\begin{pmatrix}[1.1]
    \mat{A}^{(1)}&      \\
     &\ddots   \\
     &&\mat{A}^{(T)}
\end{pmatrix},\quad
    \vec{Q}\left(\vec{u},\theta \right)=\begin{pmatrix}[1.1]
    \ \vec{q}\left(\vec{u}^{(1)}, \theta \right)\\
    \ \vdots\\
    \ \vec{q}\left(\vec{u}^{(T)}, \theta \right)
\end{pmatrix},\quad \text{and} \quad
\vec{u}=\begin{pmatrix}[1.1]
    \ \vec{u}^{(1)}\\
    \ \vdots\\
    \ \vec{u}^{(T)}
\end{pmatrix},
\end{equation}
where $\mat{A} \in \reals^{(N+1)T \times (N+1)T}$, 
$\mat{Q}\left(\vec{u},\theta\right): \reals^{(N+1)T} \times \reals_{\ge 0} \mapsto \reals^{(N+1)T}$, 
and $\vec{u} \in \reals^{(N+1)T}$.
\begin{prop}
Equation \eqref{eq:hyperbolic1d}--\eqref{eq:matA1d} is a system of {\color{black}strictly} hyperbolic partial differential equation.
\end{prop}
\begin{proof}
To prove hyperbolicity we must show that the matrix $\mat{A} \in \reals^{(N+1)T \times (N+1)T}$ in \eqref{eq:hyperbolic1d}--\eqref{eq:matA1d} is diagonalizable with only real eigenvalues.
Showing this is equivalent to showing that each block matrix $\mat{A}^{(j)} \in \reals^{(N+1)\times (N+1)}$ is diagonalizable with only real eigenvalues, since $\mat{A}$ is simply a block diagonal matrix with blocks $\mat{A}^{(j)}$ for $j=1,2,\ldots,T$. 

To show that $\mat{A}^{(j)}$ is diagonalizable with only real eigenvalues, we write the tridiagonal matrix $\mat{A}^{(j)}$ as follows and define a diagonal matrix $\mat{D}$:
\begin{equation}
    \mat{A}^{(j)}=\begin{pmatrix}
     a& b_1      \\
     c_1&a&b_2     \\
     &c_2&a&\ddots   \\
    &&\ddots&\ddots&b_{N-1}     \\
    &&&c_{N-1}&a&b_{N}\\
     &&&&c_{N}&a
    \end{pmatrix},\qquad
    \mat{D}=\begin{pmatrix}
     1&       \\
      &\sqrt{\frac{c_1}{b_1}}&    \\
     &&\sqrt{\frac{c_2}{b_2}}&   \\
    &&&\ddots&    \\
     &&&&\sqrt{\frac{c_{N}}{b_{N}}}
    \end{pmatrix},
\end{equation}
where $a=\mu_j$ and for all $k=1,2,\cdots,N$: $b_k = k/(2k-1)\neq 0$ and $c_k = (k+1)/(2k+3) \neq0$.
It can then easily be shown that $\mat{P}^{(j)}={\mat{D}}^{-1}\,\mat{A}^{(j)}\,\mat{D}$, where
\begin{equation}
    \mat{P}^{(j)}=\begin{pmatrix}
     a& \sqrt{b_1c_1} &&&&     \\
     \sqrt{b_1c_1}&a&\sqrt{b_2c_2} &&&    \\
     &\sqrt{b_2c_2}&a&\ddots &&  \\
    &&\ddots&\ddots&\sqrt{b_{N-1}c_{N-1}}&     \\
    &&&\sqrt{b_{N-1}c_{N-1}}&a&\sqrt{b_Nc_N}\\
     &&&&\sqrt{b_Nc_N}&a
    \end{pmatrix}.
\end{equation}
Now since the matrix $\mat{P}^{(j)}$ is real and symmetric, it has only real eigenvalues. Also, all the eigenvalues of $\mat{P}^{(j)}$ are distinct since all off-diagonal elements are nonzero \cite{parlett}. Hence, $\mat{A}^{(j)}$ has only real eigenvalues and all the eigenvalues are distinct by similarity to $\mat{P}^{(j)}$. \qed
\end{proof}

\begin{prop}
The eigenvalues of the matrix  $\mat{A} \in \reals^{(N+1)T \times (N+1)T}$ in \eqref{eq:hyperbolic1d}--\eqref{eq:matA1d}  are real, distinct, and given by
\begin{equation}
\lambda_{k+(j-1)(N-1)} = \mu_j + \left(\frac{\Delta \mu}{2} \right) s_k
   \quad \text{for} \quad k=1,\ldots,N+1, \quad 
   j=1,\ldots,T,
\end{equation}
where $s_k \in (-1,1)$ for $k=1,2,\ldots,N+1$ are the real distinct roots of the degree $(N+1)$ Legendre polynomial. If we assume that the $N+1$ Legendre roots are ordered so that
$s_{N+1}$ is the largest root (i.e., closest to +1), then the largest eigenvalue (in absolute value) of $\mat{A}$ is given by
\begin{equation}
\label{eqn:spectral_radius}
    \rho{(\mat{A})}=-\lambda_1 = \lambda_{(N+1)T} = \mu_T+\left(\frac{\Delta{\mu}}{2}\right) s_{N+1}.
\end{equation}
\end{prop}
\begin{proof}
It can be shown that the eigenvalues of $\mat{A}^{(j)} \in \reals^{(N+1)\times (N+1)}$ from
\eqref{eq:hn1d_2} are given by the following formula (e.g., see Cohen \cite{article:Cohen1996}):
\begin{equation}
 \lambda^{(j)}_k = \mu_j + \left(\frac{\Delta \mu}{2} \right) s_k, \quad \text{for}
 \quad  k=1,2,\ldots,N+1,
\end{equation}
where $s_k \in (-1,1)$ for $k=1,2,\ldots,N+1$ are the roots of the degree $(N+1)$ Legendre polynomial.  Since $\mat{A}$ is just a block diagonal matrix with $\mat{A}^{(j)} \in \reals^{(N+1)\times(N+1)}$ as the blocks for $j=1,2,\ldots,T$ (as shown in \eqref{eq:matA1d}), it follows that the eigenvalues of $\mat{A}$ are of the form:
\begin{equation}
\label{eqn:evalues}
   \lambda_{k+(j-1)(N+1)} = \mu_j + \left(\frac{\Delta \mu}{2} \right) s_k
   \quad \text{for} \quad k=1,\ldots,N+1, \quad
   j=1,\ldots,T.
\end{equation}

The most negative and most positive eigenvalues in \eqref{eqn:evalues} are $\lambda_1$ ($k=1$, $j=1$)
and $\lambda_{(N+1)T}$ ($k=N+1$, $j=T$), respectively. This follows directly from the assumed
ordering of $s_k$, as well as the ordering assumed in the definition of $\mu_j$ (see equation
\eqref{eqn:mu_mesh}). In fact, these two eigenvalues have the same magnitude, since by definition: $s_1 = -s_{N+1}$ and $\mu_1 = -\mu_T$. This gives the desired result: equation \eqref{eqn:spectral_radius}.
\qed
\end{proof}
{\color{black}
\begin{remark}
The importance of the formulation of the hyperbolic system in numerical simulations can be explained by two factors, stability and well-posedness. For example, the stability of the initial value problem for strongly hyperbolic systems is shown in Chapter 5 of \cite{gustaffson} and stiff well-posedness of the Cauchy problem for strongly hyperbolic system has been proved in \cite{lorenz}.
\end{remark}}
\begin{prop}
$\rho{(\mat{A})} < 1$ for any $T$ and $N$ of H$^T_N$.
\end{prop}
\begin{proof}
A simple calculation, using the fact that $s_{N+1}<1$, shows that
\begin{gather*}
\rho{(\mat{A})} = \mu_T + \left(\frac{\Delta \mu}{2} \right) s_{N+1}
= 1 - \left(\frac{\Delta \mu}{2} \right) + \left(\frac{\Delta \mu}{2} \right) s_{N+1}
< 1 - \left(\frac{\Delta \mu}{2} \right) + \left(\frac{\Delta \mu}{2} \right) = 1, \\
\Longrightarrow \quad \rho{(\mat{A})} < 1.
\end{gather*}
Furthermore, we note that $\rho{(\mat{A})} = 1 - {\mathcal O}(\Delta \mu)$ as $\Delta \mu \rightarrow 0$.
\qed
\end{proof}
{\color{black}
\begin{remark}
The H$^T_N$ model possesses the correct physical property that the propagation rate cannot exceed the speed of light, i.e., the characteristic speed is less than the speed of light.
\end{remark}}
\begin{remark}
\label{sec:remark_eigenvalue}
We have shown via the above propositions that all of the eigenvalues of $\mat{A}$ are real, distinct,  and strictly between $-1$ and $1$. In practice, we would also like to impose the condition that 
$\mat{A}$ not have a zero eigenvalue; this additional requirement is useful especially in the computation of steady-state solutions and in the imposition of inflow/outflow boundary conditions. We can always ensure that $\mat{A}$ does not have a zero eigenvalue if either of the following conditions are satisfied:
\begin{enumerate}
\item $T$ is even; or
\item $T$ is odd and $N$ is odd.
\end{enumerate} 
These assertions follow directly from definition \eqref{eqn:evalues}.
\end{remark}

\subsection{\bf The radiation energy density}\label{sec:hd1d-3}
Using the H$^T_N$ asantz \eqref{eq:hn_ansatz1d}, the angular moment of the radiation density \eqref{eq:moment1} can be
written as follows:
\begin{equation}\label{eq:moment_E}
\begin{split}
    E(t,z)&:=\frac{1}{c}\int_{-1}^{1}{I(t,z,\mu)}\diff\mu
    =\frac{1}{c}\sum_{j=1}^{T}{\color{black}\frac{\Delta{\mu}}{2}}{\int_{-1}^{1}{I^{(j)}(t,z,\alpha)}\diff\alpha}\\
      &={\color{black}\frac{1}{c}}\sum_{j=1}^{T}{\color{black}\frac{\Delta{\mu}}{2}}{\int_{-1}^{1}{\sum_{k=0}^{N} 
      \left(\frac{2k+1}{2}\right) u_k^{(j)}(t,z)p_{k}(\alpha)}}\diff\alpha
      ={\color{black}\frac{1}{c}}\sum_{j=1}^{T}{\color{black}\frac{\Delta{\mu}}{2}}{u^{(j)}_{0}(t,z)}.
\end{split}
\end{equation}
Using this result, we can rewrite material-energy equation \eqref{eq:material1d} as follows:
\begin{equation}
\label{eqn:final_mat_energy}
    {C_v} \frac{\uppartial \theta}{\uppartial t}={\color{black}{\sigma }}\sum_{j=1}^{T}{\color{black}\frac{\Delta{\mu}}{2}}{u^{(j)}_{0}(t,z)}-\sigma ac{\theta}^4.
\end{equation}

\section{Semi-implicit discontinuous Galerkin (DG) method}\label{sec:dg}

In this section we develop a discontinuous Galerkin (DG) finite element method with a semi-implicit time discretization for solving the 1D H$^T_N$ system \eqref{eq:hyperbolic1d}--\eqref{eq:matA1d} along with the material-energy equation \eqref{eqn:final_mat_energy}. 
The strategy that we employ is similar to the method developed by McClarren et al. \cite{mcclarren1} for the P$_N$ equations, but here we extend that work to the H$^T_N$ approximation as described above.

\subsection{\bf Discontinuous Galerkin spatial discretization}
\label{sec:dg1}
The H$^T_N$ approximation as described in the previous sections divides the phase space into discrete velocity bands, inside of which we approximate the specific intensity, $I(t,x,\mu)$, by a finite polynomial ansatz; this is illustrated in \figref{fig:mucells}(a). In order to spatially discretize the resulting H$^T_N$ equations: \eqref{eq:hyperbolic1d}--\eqref{eq:matA1d} and \eqref{eqn:final_mat_energy}, we additionally introduce a mesh in the $z$-coordinate:
\begin{gather}
 \left[ z_L, z_R \right] = \bigcup_{i=1}^{N_z} {\mathcal T}_i, \quad \text{where} \quad
 {\mathcal T}_i = \left[z_{i}-\frac{\Delta z}{2}, z_{i}+\frac{\Delta z}{2}\right], \\
 z_{i} = z_L + \left(i - \frac{1}{2} \right) \Delta z, \quad \text{and} \quad
 \Delta z = \frac{z_R - z_L}{N_z}.
\end{gather} 
The full $z-\mu$ phase space mesh is illustrated in \figref{fig:mucells}(b).
On each mesh element we define a local coordinate as follows:
\begin{equation}
z\Bigl|_{\Tm_i} = z_i + \left(\frac{\Delta z}{2} \right) \xi, \qquad \text{where} \quad \xi \in \left[-1,1\right].
\end{equation}

The test and trial functions for the discontinuous Galerkin scheme will come
from following {\it broken} finite element spaces:
\begin{equation}
\label{eqn:broken_space}
    \WS^{\Delta z}_{\meq} := \left\{ \vec{w}^{\Delta z} \in \Bigl[ L^{\infty}\left[ z_L, z_R \right] \Bigr]^{\meq}: \,
    \vec{w}^{\Delta z} \Bigl|_{\Tm_i} \in \left[ {\mathbb P} \left(\mdeg\right) \right]^{\meq} \, \, \forall \Tm_i \right\},
\end{equation}
where $\meq$ is the number of equations and ${\mathbb P} \left(\mdeg\right)$ is the set
of all polynomials with maximum polynomial order $\mdeg$.
On each mesh element we define the $\mdeg+1$ Gauss-Lobatto points: $\xi_j \in \left[-1,1\right]$ for $j=1,\ldots,\mdeg+1$ (e.g., see \cite{web:wikipedia_GL}). For each Gauss-Lobatto point we define the associated Lagrange interpolating polynomial:
\begin{equation}
 \Phi_j\left(\xi\right) = \prod_{\underset{k\ne j}{k=1}}^{\mdeg+1} \frac{\left( \xi - \xi_k \right)}{\left(\xi_j - \xi_k \right)}, \quad \text{s.t.} \quad
 \Phi_j\left(\xi_\ell \right) = \delta_{j\ell} \quad \text{for} \quad j,\ell = 1,\ldots,\mdeg+1.
\end{equation} 
The approximate solution on each element can then be written as 
\begin{align}
\label{eqn:dg_approx}
\Bigl\{ \vec{u}^{\Delta z}\left(t,z\left(\xi\right)\right), \,
{\theta}^{\Delta z}\left(t,z\left(\xi\right)\right) \Bigr\}
\Biggl|_{\Tm_i} = \sum_{j=1}^{\mdeg+1} \Bigl\{ \vec{U}_{ij}\left(t\right), \, \Theta_{ij}\left(t\right)
\Bigr\} \, \Phi_j\left(\xi\right), 
\end{align}
where $\vec{u}^{\Delta z}(t,z): \reals_{\ge 0} \times \reals \mapsto \WS^{\Delta z}_{(N+1)T}$,
$\theta^{\Delta z}(t,z): \reals_{\ge 0} \times \reals \mapsto \WS^{\Delta z}_{1}$,
$\Phi_j\left(\xi\right): [-1,1] \mapsto \reals$, and $\vec{U}_j(t): \reals_{\ge 0} \mapsto
\reals^{(N+1)T}$. Similarly, we write the approximate source as
\begin{equation}
\label{eqn:dg_source}
\vec{Q}^{\Delta z}\left(\vec{u}^{\Delta z}, \theta^{\Delta z} \right)\Biggl|_{\Tm_i} = \sum_{j=1}^{\mdeg+1} \vec{Q}\left(\vec{U}_{ij}\left(t\right), \Theta_{ij}(t) \right) \, \Phi_j\left(\xi\right).
\end{equation}

To obtain the spatially discretized version of H$^T_N$ system \eqref{eq:hyperbolic1d} on each element $\Tm_i$, we replace the exact solution by \eqref{eqn:dg_approx}, the exact source 
by \eqref{eqn:dg_source}, multiply \eqref{eq:hyperbolic1d} by a test function $\Phi_k(\xi)$, integrate over the element, and apply integration-by-parts in $\xi$:
\begin{equation}
\label{eqn:mass_and_der_matrices}
\begin{split}
    \frac{1}{c}\,\sum_{j=1}^{\mdeg+1} \frac{\diff\vec{U}_{ij}(t)}{\diff t} M_{kj} \, + \, \left(\frac{2}{\Delta z}\right) \left[ \Phi_k(1) \, {\mathcal F}_{i+\half}(t) 
    \, -  \, \Phi_k(-1) \, {\mathcal F}_{i-\half}(t) \right] \\ - \left(\frac{2}{\Delta z}\right) \mat{A}\sum_{j=1}^{\mdeg+1} \vec{U}_{ij}(t) N_{kj} =
    \sum_{j=1}^{\mdeg+1} \vec{Q}\left(\vec{U}_{ij}(t), \Theta_{ij}(t)\right) M_{kj},
    \end{split}
\end{equation}
for each $k=1,\ldots,\mdeg$, where we used the fact that $\uppartial_z = \left(2/\Delta z \right) \, \uppartial_{\xi}$, and where 
\begin{align}
M_{kj} = \int_{-1}^{1}\Phi_k(\xi) \, \Phi_j(\xi) \, \diff \xi \qquad \text{and} \qquad
N_{kj} = \int_{-1}^{1} \Phi_j(\xi) \, \Phi'_k(\xi) \, \diff \xi.
\end{align}
The numerical flux on each element face is defined as follows:
\begin{equation}
\begin{split}
{\mathcal F}_{i-\half}(t) &= \frac{1}{2} \mat{A} \sum_{j=1}^{\mdeg+1} \Bigl( \vec{U}_{i-1 \, j}(t) 
\, \Phi_j(1)
+ \vec{U}_{ij}(t) \, \Phi_j(-1) \Bigr) \\ &+ \frac{1}{2} \left|\mat{A}\right| \sum_{j=1}^{\mdeg+1} \Bigl( \vec{U}_{i-1 \, j}(t) \, \Phi_j(1)
- \vec{U}_{ij}(t) \, \Phi_j(-1) \Bigr),
\end{split}
\end{equation}
for $i=1,\ldots,N_z+1$.
In the above expression,  $|\mat{A}|$ is defined through the eigenvalues of $\mat{A}$:
\begin{equation}
\mat{A} = \mat{V} \mat{\Lambda} \mat{V}^{-1} \quad \Longrightarrow \quad
\left|\mat{A}\right| = \mat{V} \left|\mat{\Lambda}\right| \mat{V}^{-1},
\end{equation}
where $\mat{V}$ is the matrix of right eigenvectors of $\mat{A}$, $\mat{\Lambda}=\text{diag}\left(\lambda_1,\ldots,\lambda_{(N+1)T} \right)$ is the diagonal matrix of eigenvalues of $\mat{A}$, and $\left|\mat{\Lambda}\right| = \text{diag}\left(|\lambda_1|,\ldots,|\lambda_{(N+1)T}| \right)$.

In order to close this semi-discrete system, we also need to semi-discretize material-energy equation \eqref{eqn:final_mat_energy}. Following all of the above outlined procedures, this results in the
following equation:
\begin{equation}
    {C_v} \sum_{j=1}^{\mdeg+1} \frac{\diff \Theta_{ij}(t)}{\diff t} M_{kj} =  \sum_{j=1}^{\mdeg+1} \Bigg(\sigma \sum_{\ell=1}^{T}{\color{black}\frac{\Delta{\mu}}{2}}{U^{(\ell)}_{ij(1)}(t)}-\sigma ac\left(\Theta_{ij}(t)\right)^4\Bigg)
    M_{kj},
\end{equation}
where $U^{(\ell)}_{ij(1)}$ refers to the $(1+(\ell-1)(N+1))^{\text{th}}$ component of $\vec{U}_{ij}$
(i.e., the first component of $\vec{U}_{ij}$ in the $\ell^{\text{th}}$ velocity band).

For the remainder of this paper we consider the case $\mdeg=1$ (i.e., the piecewise linear DG approximation), which yields a second-order accurate spatial approximation. The basis functions in this case are 
\begin{equation}
\Phi_1(\xi) = \frac{1}{2} \left( 1 - \xi \right) \qquad \text{and} \qquad
\Phi_2(\xi) = \frac{1}{2} \left( 1 + \xi \right).
\end{equation}
In this case, all of the above expressions simplify greatly. For example, the values defined in
\eqref{eqn:mass_and_der_matrices} reduce to the following:
\begin{equation}
\mat{M} = \frac{1}{3}\begin{bmatrix}[1.1]
2 & 1 \\ 1 & 2
\end{bmatrix} \qquad \text{and} \qquad
\mat{N} = \frac{1}{2} \begin{bmatrix}[1.1] 
-1 & -1 \\ \, \, \, \, \, 1 & \, \, \, \, \, 1
\end{bmatrix}.
\end{equation}
After some simple algebra, we arrive at the following semi-discrete system:
\begin{align}
\frac{1}{c} \frac{\diff\vec{U}_{i1}}{\diff t}   &-  
 \frac{ 2 {\mathcal F}_{i+\half}  + 4 {\mathcal F}_{i-\half} - 
3 \mat{A} \bigl( \vec{U}_{i1} + \vec{U}_{i2} \bigr)}{\Delta z} =
    \vec{Q}\Bigl(\vec{U}_{i1}, \Theta_{i1}\Bigr), \\
    \frac{1}{c} \frac{\diff\vec{U}_{i2}}{\diff t}  &+  
 \frac{  4 {\mathcal F}_{i+\half}  + 2 {\mathcal F}_{i-\half} - 
3 \mat{A} \bigl( \vec{U}_{i1}+ \vec{U}_{i2} \bigr)}{\Delta z}  =
    \vec{Q}\Bigl(\vec{U}_{i2}, \Theta_{i2}\Bigr), \\
     {C_v} \frac{\diff \Theta_{ij}}{\diff t} &=  \sigma\sum_{\ell=1}^{T}{\color{black}\frac{ \Delta{\mu}}{2}}{U^{(\ell)}_{ij(1)}}-\sigma ac\Theta_{ij}^4, \quad \text{for} \, \, j=1,2,
\end{align}
for $i=1,\ldots,(N+1)T$. The numerical fluxes in the $\mdeg=1$ reduces to the following:
\begin{equation}
\begin{split}
{\mathcal F}_{i-\half}(t) &= \frac{1}{2} \mat{A} \bigl( \vec{U}_{i \, 1}(t) + \vec{U}_{i-1 \, 2}(t) \bigr) - \frac{1}{2} \left|\mat{A}\right| \bigl( \vec{U}_{i \, 1}(t) - \vec{U}_{i-1 \, 2}(t) \bigr),
\end{split}
\end{equation}
for $i=1,\ldots,N_z+1$.

\subsection{\bf Semi-implicit time scheme: nonlinear version}\label{sec:dg-2}
For the time integration we adopt the semi-implicit scheme by McClarren et al. \cite{mcclarren1}, which
 is a two-stage Runge-Kutta method.
The first stage (i.e., the predictor step) can be written as follows:
\begin{align}
\label{eq:predictor1}
\frac{1}{c} \frac{\vec{U}^{n+\half}_{i1}-\vec{U}^{n}_{i1}}{\Delta t/2}   &=  
 +\frac{ 2 {\mathcal F}_{i+\half}^n  + 4 {\mathcal F}_{i-\half}^n - 
3 \mat{A} \bigl( \vec{U}_{i1}^n + \vec{U}_{i2}^n \bigr)}{\Delta z}  +
    \vec{Q}\Bigl(\vec{U}^{n+\half}_{i1}, \Theta_{i1}^{n+\half} \Bigr), \\
\label{eq:predictor2}
    \frac{1}{c} \frac{\vec{U}^{n+\half}_{i2}-\vec{U}^{n}_{i2}}{\Delta t/2}  &=  
- \frac{  4 {\mathcal F}^n_{i+\half}  + 2 {\mathcal F}^n_{i-\half} - 
3 \mat{A} \bigl( \vec{U}^n_{i1}+ \vec{U}^n_{i2} \bigr)}{\Delta z}  +
    \vec{Q}\Bigl(\vec{U}_{i2}^{n+\half}, \Theta_{i2}^{n+\half} \Bigr), \\
\label{eq:predictor3}
    {C_v} \frac{\Theta^{n+\half}_{ij}-\Theta^{n}_{ij}}{\Delta t/2} &=  \sigma\sum_{\ell=1}^{T}{\color{black}\frac{ \Delta{\mu}}{2}}{U^{(\ell) \, n+\half}_{ij(1)}}-\sigma ac\left(\Theta_{ij}^{n+\half} \right)^4, 
     \quad \text{for} \, \, j=1,2.
\end{align}
Similarly, the second stage (i.e., the correction step) can be written as follows:
\begin{align}
\label{eq:corrector1}
\frac{1}{c} \frac{\vec{U}^{n+1}_{i1}-\vec{U}^{n}_{i1}}{\Delta t}   &=  
+ \frac{ 2 {\mathcal F}_{i+\half}^{n+\half}  + 4 {\mathcal F}_{i-\half}^{n+\half} - 
3 \mat{A} \left( \vec{U}_{i1}^{n+\half} + \vec{U}_{i2}^{n+\half} \right)}{\Delta z}  +
    \vec{Q}\Bigl(\vec{U}^{n+1}_{i1}, \Theta_{i1}^{n+1} \Bigr), \\
\label{eq:corrector2}
    \frac{1}{c} \frac{\vec{U}^{n+1}_{i2}-\vec{U}^{n}_{i2}}{\Delta t}  &=  
- \frac{  4 {\mathcal F}^{n+\half}_{i+\half}  + 2 {\mathcal F}^{n+\half}_{i-\half} - 
3 \mat{A} \left( \vec{U}^{n+\half}_{i1}+ \vec{U}^{n+\half}_{i2} \right)}{\Delta z}  +
    \vec{Q}\Bigl(\vec{U}_{i2}^{n+1}, \Theta_{i2}^{n+1} \Bigr), \\
\label{eq:corrector3}
    {C_v} \frac{\Theta^{n+1}_{ij}-\Theta^{n}_{ij}}{\Delta t} &=  \sigma\sum_{\ell=1}^{T}{\color{black}\frac{ \Delta{\mu}}{2}}{U^{(\ell) \, n+1}_{ij(1)}}-\sigma ac\left(\Theta_{ij}^{n+1} \right)^4, 
     \quad \text{for} \, \, j=1,2.
\end{align}
The numerical flux in both the first and second stages is of the following form:
\begin{equation}
\label{eqn:discrete_numerical_flux}
{\mathcal F}_{i-\half}^{\star} = \frac{1}{2} \mat{A} \bigl( \vec{U}^{\star}_{i \, 1} + \vec{U}_{i-1 \, 2}^{\star} \bigr) - \frac{1}{2} \left|\mat{A}\right| \bigl( \vec{U}_{i \, 1}^{\star} - \vec{U}_{i-1 \, 2}^{\star} \bigr),
\end{equation}
where $\star \in \left\{ n, n+\half \right\}$.

The time-stepping scheme described above is semi-implicit in that it is explicit on the wave propagation terms and implicit on the collision terms. Since in the TRT system the collision source is a nonlinear function of the temperature, the result is that in each stage a nonlinear algebraic equation must be solved. In order to avoid this, we show in the next section how to linearize the source.

\subsection{\bf Semi-implicit time scheme: linearized version}\label{sec:dg-3}
The scheme we propose in this work for solving the 1D H$^T_N$ system \eqref{eq:hyperbolic1d}--\eqref{eq:matA1d} along with the material-energy equation \eqref{eqn:final_mat_energy} is a variant of the semi-implicit scheme described by
\eqref{eq:predictor1}--\eqref{eq:corrector3}, but with the additional feature that the source
is linearized. The proposed scheme is a H$^T_N$ extension of the scheme developed for the P$_N$ system in \cite{mcclarren1}. The details of this scheme are provided in this section.

We begin this discussion by recalling that source, $\vec{Q}$, in  
\eqref{eq:predictor1}--\eqref{eq:predictor2} and \eqref{eq:corrector1}--\eqref{eq:corrector2}
can be written as follows:
\begin{gather}
\vec{Q}\left(\vec{U}^{\star}_{ij},\Theta^{\star}_{ij} \right)= -\sigma\vec{U}^{\star}_{ij} + 
\sigma  ac\left( \Theta_{ij}^{\star} \right)^4 \vec{\tilde{e}},
\end{gather}
where $\star \in \left\{ n+\half,n+1 \right\}$ and $\vec{\tilde{e}} \in \reals^{(N+1)T}$ 
is a vector with the following components:
\begin{equation}
\label{eqn:special_e_vec}
{\tilde{e}}_{m} = \begin{cases} 1 & \text{if} \quad m = 1+(k-1)(N+1) \quad \text{for} \quad k=1,2,\ldots,T, \\
0 & \text{otherwise}.
\end{cases}
\end{equation}
These sources are clearly nonlinear in the temperature $\Theta$. In order to linearize $\vec{Q}$ we invoke the following two Taylor expansions in temperature:
\begin{align}
\left(\Theta^{n+\half}_{ij} \right)^4 &= \left(\Theta^{n}_{ij} \right)^4 + 4 \left(\Theta^{n}_{ij} \right)^3 \left( \Theta^{n+\half}_{ij} - \Theta^{n}_{ij} \right) +
{\mathcal O}\left( \left( \Theta^{n+\half}_{ij} - \Theta^{n}_{ij} \right)^2 \right), \\
\left(\Theta^{n+1}_{ij} \right)^4 &= \left(\Theta^{n}_{ij} \right)^4 + 4 \left(\Theta^{n}_{ij} \right)^3 \left( \Theta^{n+1}_{ij} - \Theta^{n}_{ij} \right) +
{\mathcal O}\left( \left( \Theta^{n+1}_{ij} - \Theta^{n}_{ij} \right)^2 \right).
\end{align}
Since our overall method is only accurate to second order, we can safely disregard the second order corrections in the above expressions. Furthermore, the linear differences in the Taylor expansions can be replaced via the update formulas \eqref{eq:predictor3} and \eqref{eq:corrector3}, respectively, yielding:
\begin{align}
\label{eqn:theta_lin1}
\left(\Theta^{n+\half}_{ij} \right)^4 &\approx \left(\Theta^{n}_{ij} \right)^4 + 4 \left(\Theta^{n}_{ij} \right)^3 \frac{\Delta t}{2 C_v}\left(\sigma\sum_{\ell=1}^{T}{\color{black}\frac{ \Delta{\mu}}{2}}{U^{(\ell) \, n+\half}_{ij(1)}}-\sigma ac\left(\Theta_{ij}^{n+\half} \right)^4 \right), \\
\label{eqn:theta_lin2}
\left(\Theta^{n+1}_{ij} \right)^4 &\approx \left(\Theta^{n}_{ij} \right)^4 + 4 \left(\Theta^{n}_{ij} \right)^3 \frac{\Delta t}{C_v}\left( \sigma\sum_{\ell=1}^{T}{\color{black}\frac{ \Delta{\mu}}{2}}{U^{(\ell) \, n+1}_{ij(1)}}-\sigma ac\left(\Theta_{ij}^{n+1} \right)^4 \right).
\end{align}
Treating the above approximations as equalities and solving \eqref{eqn:theta_lin1} and \eqref{eqn:theta_lin2} for the fourth power
of $\Theta^{n+\half}_{ij}$ and $\Theta^{n+1}_{ij}$, respectively, yields:
\begin{align}
\label{eqn:theta_four_lin1}
\left(\widetilde\Theta^{n+\half}_{ij} \right)^4 :=& \frac{\left(\Theta_{ij}^n \right)^3 \left(\frac{C_v}{\sigma} \Theta_{ij}^n +  {\color{black}2}\Delta t  \displaystyle\sum_{\ell=1}^{T}{\color{black}\frac{ \Delta{\mu}}{2}}{U^{(\ell) \, n+\half}_{ij(1)}} \right)}{\frac{C_v}{\sigma} + 2 \Delta t  a c  \left( \Theta_{ij}^n \right)^3 }, \\
\label{eqn:theta_four_lin2}
\left(\widetilde\Theta^{n+1}_{ij} \right)^4 :=& \frac{\left(\Theta_{ij}^n \right)^3 \left(\frac{C_v}{\sigma} \Theta_{ij}^n +  {\color{black}4}\Delta t  \displaystyle\sum_{\ell=1}^{T}{\color{black}\frac{ \Delta{\mu}}{2}}{U^{(\ell) \, n+1}_{ij(1)}} \right)}{\frac{C_v}{\sigma} + 4 \Delta t  a c  \left( \Theta_{ij}^n \right)^3 }.
\end{align}
Using these versions of the fourth power of the temperature successfully linearizes the source terms in \eqref{eq:predictor1}-\eqref{eq:predictor2} and \eqref{eq:corrector1}-\eqref{eq:corrector2}.

In order to complete the linearization of the source terms, we now turn our attention to
\eqref{eq:corrector3}.  In particular, we replace the fourth power of $\Theta_{ij}^{n+1}$ in 
\eqref{eq:corrector3} by
\eqref{eqn:theta_four_lin2}. After some simple algebra, we now arrive at the final semi-implicit discontinuous Galerkin scheme that is advocated in this work. The first stage is
\begin{align}
\label{eq:predictor1_linearized}
\frac{\vec{U}^{n+\half}_{i1}-\vec{U}^{n}_{i1}}{\sigma c \left({\Delta t}/{2}\right)}   &=  
 +\frac{ 2 {\mathcal F}_{i+\half}^n  + 4 {\mathcal F}_{i-\half}^n - 
3 \mat{A} \bigl( \vec{U}_{i1}^n + \vec{U}_{i2}^n \bigr)}{\sigma \Delta z}  
-\vec{U}^{n+\half}_{i1} +   ac\left( \widetilde\Theta_{i1}^{n+\half} \right)^4 \vec{\tilde{e}}, \\
\label{eq:predictor2_linearized}
    \frac{\vec{U}^{n+\half}_{i2}-\vec{U}^{n}_{i2}}{\sigma c \left({\Delta t}/{2}\right)}  &=  
- \frac{  4 {\mathcal F}^n_{i+\half}  + 2 {\mathcal F}^n_{i-\half} - 
3 \mat{A} \bigl( \vec{U}^n_{i1}+ \vec{U}^n_{i2} \bigr)}{\sigma \Delta z}  
-\vec{U}^{n+\half}_{i2} +   ac\left( \widetilde\Theta_{i2}^{n+\half} \right)^4 \vec{\tilde{e}},
    \end{align}
where $\widetilde\Theta^{n+\half}_{ij}$ is defined by \eqref{eqn:theta_four_lin1}
and $\vec{\tilde{e}}$ is defined by \eqref{eqn:special_e_vec}. 
The second stage is
\begin{align}
\label{eq:corrector1_linearized}
\begin{split}
\frac{\vec{U}^{n+1}_{i1}-\vec{U}^{n}_{i1}}{\sigma c \Delta t}   &=  
+ \frac{ 2 {\mathcal F}_{i+\half}^{n+\half}  + 4 {\mathcal F}_{i-\half}^{n+\half} - 
3 \mat{A} \left( \vec{U}_{i1}^{n+\half} + \vec{U}_{i2}^{n+\half} \right)}{\sigma \Delta z} 
-\vec{U}^{n+1}_{i1} +   ac\left( \widetilde\Theta_{i1}^{n+1} \right)^4 \vec{\tilde{e}}, 
\end{split}\\
\label{eq:corrector2_linearized}
\begin{split}
    \frac{\vec{U}^{n+1}_{i2}-\vec{U}^{n}_{i2}}{\sigma c \Delta t}  &=  
- \frac{  4 {\mathcal F}^{n+\half}_{i+\half}  + 2 {\mathcal F}^{n+\half}_{i-\half} - 
3 \mat{A} \left( \vec{U}^{n+\half}_{i1}+ \vec{U}^{n+\half}_{i2} \right)}{\sigma \Delta z} 
-\vec{U}^{n+1}_{i2} +   ac\left( \widetilde\Theta_{i2}^{n+1} \right)^4 \vec{\tilde{e}}, 
\end{split}\\
\label{eq:corrector3_linearized}
     \Theta_{ij}^{n+1}&=\Theta_{ij}^{n}+\frac{\Delta{t}\left[\displaystyle\sum_{\ell=1}^{T}{\color{black}\frac{\Delta{\mu}}{2}}{U_{ij(1)}^{(\ell)n+1}-ac\left(\Theta_{ij}^n\right)^{4}}\right]}{\frac{C_v}{\sigma}+4\Delta{t}ac\left( \Theta_{ij}^n \right)^3}, \quad \text{for} \, \, j=1,2,
\end{align}
where $\widetilde\Theta^{n+1}_{ij}$ is defined by \eqref{eqn:theta_four_lin2}
and $\vec{\tilde{e}}$ is defined by \eqref{eqn:special_e_vec}. Note that by the time we
reach  \eqref{eq:corrector3_linearized}, $\vec{U}^{n+1}_{ij}$ is already known, meaning that this step has the computational complexity of an explicit update. Again, the numerical fluxes in both the first and second stages are of the form \eqref{eqn:discrete_numerical_flux}.

\subsection{\bf Slope limiter}\label{sec:dg-4}
In order to remove unphysical oscillations from the numerical method described above, we include a slope limiter. McClarren and Lowrie \cite{mcclarren2} pointed out that the {\it double minmod} slope limiter, also known as the {\it monotonized central} slope limiter, is asymptotic-preserving for hyperbolic systems with stiff relaxation terms while the minmod limiter does not preserve the asymptotic limit due to discontinuities at the cell edge. Therefore, in this work we use the same double minmod limiter to preserve asymptotic limit. In particular, after each predictor and corrector step we compute the cell average:
\begin{equation}
    \overline{\vec{U}}_i=\frac{\vec{U}_{i1}+\vec{U}_{i2}}{2},
\end{equation}
and then modify the original nodal values as
\begin{equation}
    \vec{U}_{i1}:=\overline{\vec{U}}_i-\frac{\vec{s}_i}{2} \qquad \text{and} \qquad \vec{U}_{i2}:= \overline{\vec{U}}_i+\frac{\vec{s}_i}{2},
\end{equation}
where
\begin{equation}
    s_{i (\ell)}=\text{mm}\Bigl( \, U_{i2 (\ell)}-U_{i1 (\ell)},\,\,\alpha \left( \overline{{U}}_{i(\ell)}-\overline{U}_{i-1 (\ell)} \right),\,\,\alpha \left(\overline{U}_{i+1 (\ell)}-\overline{{U}}_{i(\ell)} \right)\Bigr),
\end{equation}
where $\alpha\in [0,2]$, $i=1,\cdots,N_z$, $l=1,\cdots,T(N+1)$, and the minmod function is defined as follows:
\begin{equation}
    \text{mm}(a,b,c):= 
\begin{cases}
    \text{sign}(a)\,\text{min}\left(|a|,|b|,|c|\right),& \text{if  } \text{sign}(a)=\text{sign}(b)=\text{sign}(c),\\
    0,& \text{otherwise}.
\end{cases}
\end{equation}
As explained in \cite{mcclarren1}, $\alpha=0$ is the first-order upwind or Godunov
scheme, $\alpha=1$ is the minmod limiter, and $\alpha=2$ is the monotonized central (MC) or double minmod limiter. We use $\alpha=2$ in all our numerical tests.
\subsection{\bf Boundary conditions}\label{sec:dg-5}
To complete the numerical methods section, we briefly explain how boundary conditions are implemented.
The three types of boundary conditions considered in this work in various examples are reflective,
Dirichlet, and vaccuum conditions.
In all cases we prescribe the intensities on the left and right boundaries via the following expressions: 
\begin{equation}
    I^L(t,\mu)= 
\begin{cases}
    I_{\text{out}}(t,z_L,\mu) & \text{if  } \mu>0,\\
    I(t, z_L, \mu) & \text{if  } \mu<0,
\end{cases} \quad 
    I^R(t,\mu)= 
\begin{cases}
    I(t,z_R,\mu) & \text{if  } \mu>0,\\
    I_{\text{out}}(t,z_R,\mu) & \text{if  } \mu<0,
\end{cases}
\end{equation}
respectively, where
\begin{equation}
    I_{\text{out}}(t,z,\mu)= 
\begin{cases}
    I(t, z, -\mu) & \text{if reflective BC},\\
    I(t, z,  \mu) & \text{if Dirichlet BC},\\
    0 & \text{if vacuum BC}.
\end{cases}
\end{equation}
Note that $z_L$ and $z_R$ denote the left and right boundaries, respectively.
{\color{black}
\subsection{\bf Asymptotic analysis}\label{sec:dg-6}
The grey transport equation \eqref{eq:frequencyfree_grey1d} 
and material-energy equation \eqref{eq:material1d} reduce to the so-called
{\it equilibrium diffusion limit} under a certain rescaling of the underlying parameters. The rescaled parameters are as follows:
\begin{equation}
\label{eqn:ap_rescaling}
\sigma \rightarrow \frac{\sigma}{\varepsilon}, \quad
c \rightarrow \frac{c}{\varepsilon}, \quad
C_v \rightarrow \varepsilon \, C_v, \quad
a \rightarrow \varepsilon \, a,
\end{equation}
where $\varepsilon>0$, which results in the following rescaled transport and material-energy equation (where we have set $s\equiv0$):
\begin{gather}\label{eq:frequencyfree_grey1d_rescaled}
    \frac{\varepsilon^2}{c}\frac{\uppartial I}{\uppartial t} + \varepsilon \, \mu \frac{\uppartial I}{\uppartial z}+ \sigma I=\frac{1}{2} \sigma {ac{\theta}^4}, \\
\label{eq:material1d_rescaled}
    \varepsilon^2 C_v \frac{\uppartial \theta}{\uppartial t}= \sigma \left( \int_{-1}^{1} I(t,z,\mu) \, \diff\mu -a c {\theta}^4 \right).
\end{gather}
As shown in Larsen et al. \cite{article:Larsen1983}, the highly collisional limit, $\varepsilon \rightarrow 0^+$, results in the following nonlinear diffusion equation:
\begin{equation}\label{eq:equilibrium_diffusion}
     \frac{\uppartial}{\uppartial t} \left[ C_v \theta^{(0)} + a\left(\theta^{(0)}\right)^4 \right] = \frac{\uppartial}{\uppartial z}\left[ \frac{ac}{3\sigma}\frac{\uppartial}{\uppartial z}\left(\theta^{(0)}\right)^4 \right],
\end{equation}
where $\theta^{(0)}$ refers to the leading order term in a power series expansion of the temperature $\theta$ in $\varepsilon$.

A numerical method for system \eqref{eq:frequencyfree_grey1d_rescaled}--\eqref{eq:material1d_rescaled} is called {\it asymptotic-preserving} if
for fixed discretization parameters ($\Delta t$, $\Delta z$, and $\Delta \mu$),
the numerical method in the limit $\varepsilon \rightarrow 0^+$ reduces to a 
consistent and stable discretization of \eqref{eq:equilibrium_diffusion}.

\begin{proposition}[Asymptotic preserving (AP) property in the
equilibrium diffusion limit]
The numerical method given by \eqref{eq:predictor1_linearized}--\eqref{eq:corrector3_linearized}
with \eqref{eqn:discrete_numerical_flux}, \eqref{eqn:special_e_vec}, \eqref{eqn:theta_four_lin1}, and \eqref{eqn:theta_four_lin2}, after
rescaling \eqref{eqn:ap_rescaling}, produces the following
consistent and stable discretization of \eqref{eq:equilibrium_diffusion} in the limit as $\varepsilon \rightarrow 0^+$ when the discretization parameters ($\Delta t$, $\Delta z$, and $\Delta \mu$) are held constant:
%
%
\begin{equation}\label{eq:correct_diffusion_discretization}
\begin{gathered}
    \frac{C_v\check{\Theta}^{[0]n+1}_{i2}+a\left(\check{\Theta}^{[0]n+1}_{i2}\right)^4-C_v\check{\Theta}^{[0]n}_{i2}-a\left(\check{\Theta}^{[0]n}_{i2}\right)^4}{\Delta{t}} \\ =ac\frac{\left(\Theta^{[0]n}_{i+1\,2}\right)^4-2\left(\Theta^{[0]n}_{i2}\right)^4+\left(\Theta^{[0]n}_{i1}\right)^4}{3\sigma\Delta{z}^2},
    \end{gathered}
\end{equation}
where the superscript $[s]$ represents the $s^{\text{th}}$ term of the expansion in $\varepsilon$ and the weighted average $\check{(\cdot)}$ is defined as follows:
\begin{equation}
\begin{gathered}
\check{(\cdot)}_{i2}:=\frac{1}{2}\left[\left(\bar{(\cdot)}_{i+1(1)}-\frac{1}{3}\hat{(\cdot)}_{i+1(1)}\right)+\left(\bar{(\cdot)}_{i(1)}+\frac{1}{3}\hat{(\cdot)}_{i(1)}\right)\right], \\
    \bar{(\cdot)}_i:=\frac{1}{2}\left[(\cdot)_{i1}+(\cdot)_{i2}\right],\quad
    \hat{(\cdot)}_i:=\frac{1}{2}\left[(\cdot)_{i2}-(\cdot)_{i1}\right].
    \end{gathered}
\end{equation}
We note that one obtains the same limiting numerical scheme independent of the number of velocity bands, $T$, and that this limiting numerical scheme is the same as the one achieved by the P$_N$ approximation \cite{mcclarren1}.
\end{proposition}

\begin{proof}
The proof of the above claim follows exactly the same steps as the one given in
McClarren et al. \cite{mcclarren1} for the P$_N$ system with only one minor modification: for the proposed method considered in this work, we need to sum 
update equations \eqref{eq:predictor1_linearized}--\eqref{eq:corrector2_linearized} 
over all velocity bands ($\frac{\Delta \mu}{2} \sum_{\ell=1}^T$)
before applying the perturbation series analysis.
\qed
\end{proof}
}
.pdf

\section{Numerical results}\label{sec:numerical_result}
In this section, numerical results for six standard benchmark problems for the TRT system are provided, including examples in the optically thin and thick regimes:
(1) bilateral inflow (\S\ref{sec:bilater_inflow_problem}), (2) streaming in a vaccuum (\S\ref{sec:vacuum_problem}), (3) Su-Olson problem (\S\ref{sec:su_olson_problem}), (4) diffusive Marshak wave (\S\ref{sec:diffusive_Marshak_problem}), (5) Marshak wave in thin medium (\S\ref{sec:thin_Marshak_problem}), 
and (6) smooth Marshak wave problem (\S\ref{sec:smooth_Marshak_problem}).
{\color{black}Note that the problems (1) and (2) are free streaming.}
The material-energy coupling
equation \eqref{eqn:final_mat_energy} is only required for in the examples described in \S \ref{sec:su_olson_problem}--\S\ref{sec:smooth_Marshak_problem}.

Unless otherwise stated, we choose a CFL condition for all simulations as follows:
\begin{equation}
    \text{CFL}:=\frac{c\,\rho({\mat{A}})\,\Delta{t}}{\Delta{x}}\leq 0.3,
\end{equation}
where $\rho({\mat{A}}) \approx 1$ is the spectral radius of a matrix $\mat{A}$ as
defined in equation \eqref{eqn:spectral_radius}. For numerical examples in which the diffusion dominates, for example problems like the diffusive Marshak wave problem, we are able to achieve very relaxed CFL number between 2 and 3.

\subsection{\bf Bilateral inflow}\label{sec:bilater_inflow_problem}
This problem is used to test the H$^T_N$ scheme without  the opacity $\sigma$ and the external source $s$; this shows how well the hybrid discrete approximation of the free-streaming kinetic operator and its corresponding numerical discretization captures the correct wave speeds and resolves discontinuities. In this setting, equation \eqref{eq:frequencyfree_grey1d} reduces to the following:
\begin{equation}
\frac{1}{c}\,\frac{\uppartial I}{\uppartial t}+\mu\,\frac{\uppartial I}{\uppartial z}=0,
\end{equation}
for which the analytic solution can easily be computed from the method of characteristics (e.g., see  \cite{fan}):
\begin{equation}
\label{eqn:bilateral_exact1}
I(t,z,\mu)=I^0(z-c\mu t,\mu).
\end{equation}
We choose the initial condition as
\begin{equation}
\label{eqn:bilateral_exact2}
    I^{0}(z,\mu)= 
\begin{cases}
    ac\delta(\mu-1) & \text{if } \quad z\leq0.2,\\
    0 & \text{if } \quad 0.2<z<0.8,\\
    0.5 ac & \text{if } \quad z>0.8,
\end{cases}
\end{equation}
where $\delta(\mu-1)$ is a Dirac delta centered at $\mu=1$, $c=3\times10^{10}$ cm s$^{-1}$ is the speed of light, and $a=1.372\times10^{14}$ ergs cm$^{-3}$ keV$^{-4}$ is the radiation constant.
The exact angular moment of the radiation intensity, $E(t,z)$, is given by the following if $0<ct<0.3$:
\begin{equation}
a^{-1} E(t,z)= 
\begin{cases}
    1& \text{if }\quad z\leq0.2+ct,\\
    0& \text{if } \quad 0.2+ct < z \le 0.8 -ct,\\
    \left({z-0.8+ct}\right)/\left({2ct}\right) & \text{if } \quad 0.8-ct\le z \le 0.8+ct, \\
    1 & \text{if } \quad z\ge0.8+ct.
\end{cases}
\end{equation}
We run the code until $c t_{\text{end}}=0.1$, in the physical domain $z \in [0,1]$ with inflow/outflow boundary conditions based on the exact solution \eqref{eqn:bilateral_exact1}--\eqref{eqn:bilateral_exact2}
. 
This example is challenging for moment closure methods such as P$_N$ and H$^T_N$ due to the fact that there is a delta function in velocity, as well as discontinuities in both $z$ and $\mu$. Furthermore, in this example there are no collisions to help smooth out the solution.

The scaled angular moment of the radiation intensity, $E(t,z)/a$, for various P$_N$ and H$^T_N$ approximations are shown in \figref{fig:HT_bilateral_inflow}. Each panel in \figref{fig:HT_bilateral_inflow} shows
solutions with models that have the same DOFs: 
(a) 4 moments: P$_3$ and H$^2_1$,
(b) 6 moments: P$_5$, H$^3_1$, and H$^2_2$,
(c) 8 moments: P$_7$, H$^4_1$, and H$^2_3$, and 
(d) 24 moments: P$_{23}$, H$^{12}_1$, and H$^4_5$. In each case, the H$^T_N$ model with $T>1$ gives less oscillation than the P$_N$ model with the same DOFs.

In all the simulations
shown in  \figref{fig:HT_bilateral_inflow} we have used the double minmod limiter to control unphysical oscillations and to remove negative density values; without the double minmod limiters active, both P$_N$ and H$^T_N$ solutions suffer from negative densities near the discontinuity.  In order to show how the double minmod limiter affects the solution we also provide \figref{fig:HT_bilateral_inflow_limiter}, in which we show a direct comparison of the unlimited and limited H$^4_3$ solutions.


\begin{figure}[!ht] 
\centering
\captionsetup[subfigure]{position=top,aboveskip=1pt}
 \subcaptionbox{Bilateral inflow: P$_3$ vs H$^T_N$}{\includegraphics[width=57mm]{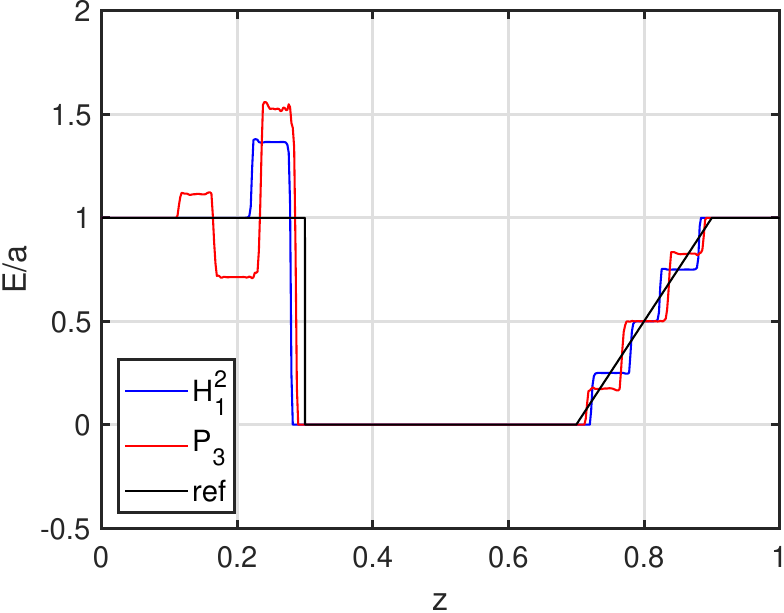}}\hfill
 \subcaptionbox{Bilateral inflow: P$_5$ vs H$^T_N$}{\includegraphics[width=57mm]{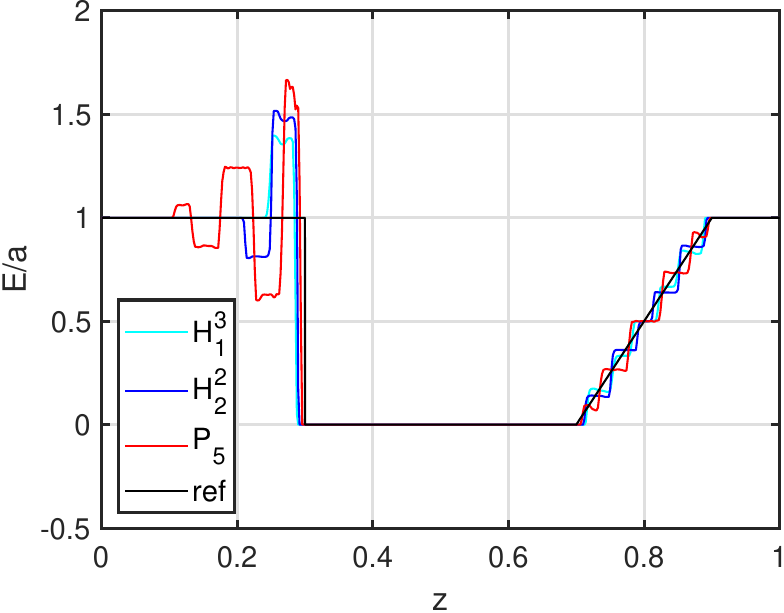}}\vspace{3mm}
 \subcaptionbox{Bilateral inflow: P$_7$ vs H$^T_N$}{\includegraphics[width=57mm]{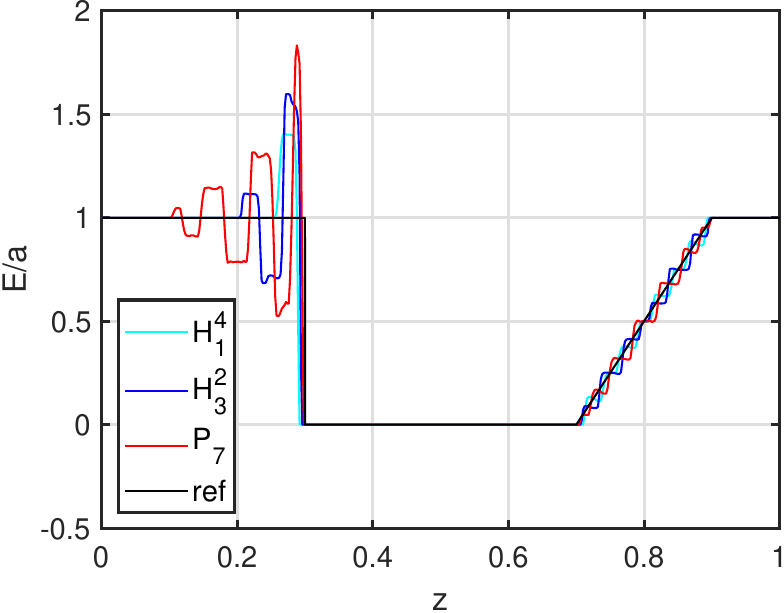}}\hfill
 \subcaptionbox{Bilateral inflow: P$_{23}$ vs H$^T_N$}{\includegraphics[width=57mm]{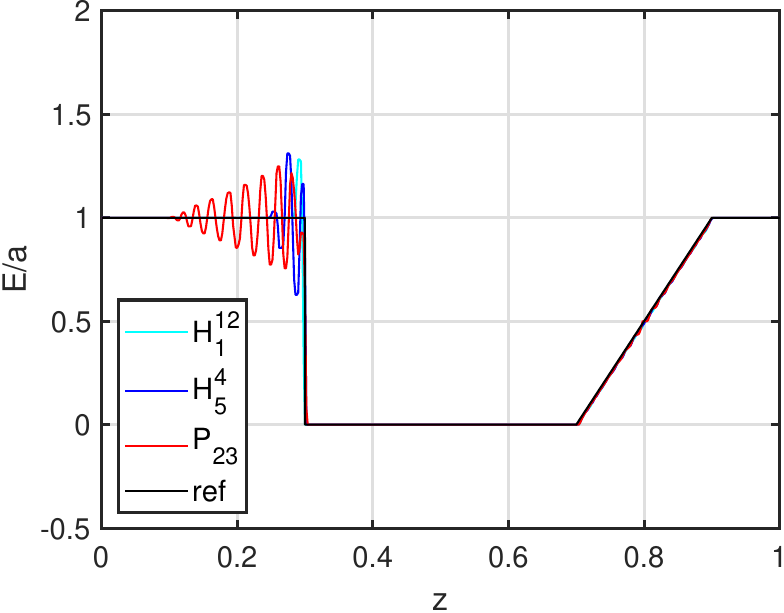}}
  \caption{(\S\ref{sec:bilater_inflow_problem}: Bilateral inflow)  Comparisons of P$_N$ and H$^T_N$ solutions for the bilateral inflow problem at $ct=0.1$, $N_z=500$, with $\text{CFL}=0.3$.}
 \label{fig:HT_bilateral_inflow}
 \end{figure}

\begin{figure}[!ht] 
  \centering
\includegraphics[width=70mm]{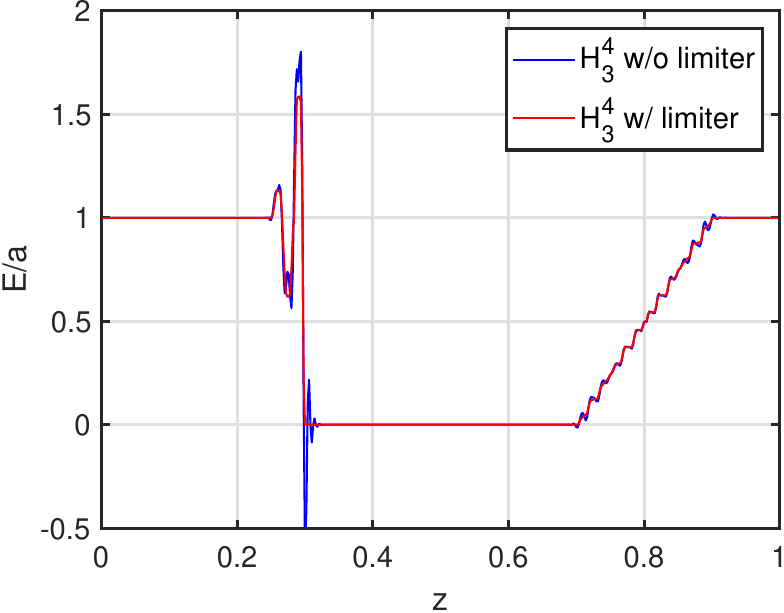}
 \caption{(\S\ref{sec:bilater_inflow_problem}: Bilateral inflow)  H$^4_3$ solution with and without the double minmod limiter for the bilateral inflow problem at $c t=0.1$, $N_z=500$, with $\text{CFL}=0.1$.}
 \label{fig:HT_bilateral_inflow_limiter}
 \end{figure}
 
\subsection{\bf Free streaming in a vacuum}\label{sec:vacuum_problem}
In this section, we  test  our  scheme  on the propagation of photons in a vacuum.
We choose a zero initial condition, $I^{0}(z,\mu)=0$, on the computational domain $z\in[0,1]$ with
 the following Dirichlet boundary conditions:
\begin{equation}
   I\left(t,z=0,\mu\right)= ac \qquad \text{and} \qquad  I\left(t,z=1,\mu\right)=0.
\end{equation}

The analytic solution for $t>0$ with the given initial and boundary conditions can be obtained by the method of characteristics as in the previous bilateral problem:
\begin{equation}
I\left(t,z,\mu \right) = 
\begin{cases}
ac & \text{if} \quad  z/(ct) < \mu \le 1, \\ 
0 & \text{if} \quad -1 \le \mu \le z/(ct).
\end{cases}
\end{equation}
The exact solution for the angular moment for $t>0$ in this case is
\begin{equation}
a^{-1} E(t,z)= \int_{z/(ct)}^{1} \diff\mu = 1-z/(ct).
\end{equation}

In \figref{fig:HTN_vacuum_conv_rate} and \tabref{tab:error_table_dof_16} we compute the absolute $L_2$ error in the scaled angular moments, $E(t,z)/a$, for various H$_N^T$ approximations:
\begin{align}
L_2\text{  error} \, :=& \, \left[\frac{1}{a N_z}\sum_{i=1}^{N_z}\Bigl( E^T_N\left(t_{\text{end}}, z_i\right)
- E_{\text{exact}}\left(t_{\text{end}}, z_i\right) \Bigr)^{2}\right]^{1/2}, \\
L_{\infty}\text{  error} \, :=& \,  \max_{1\le i \le N_z} \, \biggl| \frac{1}{a} \Bigl( E^T_N\left(t_{\text{end}}, z_i\right)
- E_{\text{exact}}\left(t_{\text{end}}, z_i\right) \Bigr) \biggr|,
\end{align}
where $E^T_N$ is the angular moment solution calculated using the H$^T_N$ approximation.
In particular, we show in \figref{fig:vacuum_conv_a} the $L_2$ convergence of H$_N^T$ with increasing $N$ 
and in \figref{fig:vacuum_conv_b} the $L_2$ convergence of H$_N^T$ with increasing $T$. \figref{fig:vacuum_conv_a} shows rapid convergence in terms of $N$ for the H$^2_N$ and H$^3_N$ solutions, however P$_N$=(H$^1_N$) shows a much slower convergence rate. Meanwhile, \figref{fig:vacuum_conv_b} shows rapid convergence with a much steeper slope than \figref{fig:vacuum_conv_a}, which suggests that increasing $T$ is a better way to achieve the desired accuracy than increasing $N$ when a discontinuity exists in the
underlying intensity $I(t,z,\mu)$. This observation agrees with the numerical values shown in \tabref{tab:error_table_dof_16}, where the $L_{2}$ and $L_{\infty}$ errors are shown
for various methods that all have the same degrees of freedom ($\text{DOF}=16$).

Additionally, we study the profile of $E(t,z)$ for various methods in \figref{fig:HTN_vacuum}.
Each H$^T_N$ solution in \figref{fig:HTN_vacuum} shows the propagation of multiple waves depending on the number of distinct eigenvalues of the matrix $\mat{A}$ defined in \eqref{eq:matA1d_sub} and \eqref{eq:matA1d}. 
In particular, in \figref{fig:vacuum_a} and \figref{fig:vacuum_b} we compare H$^1_N$ and H$^2_N$ solutions  with various $N$, respectively. Analogously, in \figref{fig:vacuum_c} and \ref{fig:vacuum_d} we compare various H$^T_1$ and H$^T_2$ solutions with different $T$, respectively. Finally, in   \figref{fig:vacuum_e} through \ref{fig:vacuum_h} we plot in each panel a different H$^T_N$ method with $\text{DOF}=16$:
(e) H$_{15}^1$, (f) H$_7^2$, (f) H$_3^4$, and
(g) H$_1^8$. Again, we conclude from these simulations that in the case when the intensity is discontinuous, increasing $T$ is more effective than increasing $N$.
 {\color{black}We also demonstrate in these panels that when $T$ is odd and $N$ is even (i.e., $\mat{A}$ has a zero eigenvalue -- see Remark \ref{sec:remark_eigenvalue}), the solutions show incorrect boundary values:
 \figref{fig:vacuum_a} (H$^1_2$ and H$^1_4$)  and  in \figref{fig:vacuum_d} (H$^1_2$ and H$^3_2$).}



\begin{figure}[htbp]
\centering
\captionsetup[subfigure]{position=top,aboveskip=1pt}
 \subcaptionbox{H$^{1}_{N}$(=P$_N$) vs H$^{2}_{N}$ vs H$^{3}_{N}$ \label{fig:vacuum_conv_a}}{\includegraphics[width=57mm]{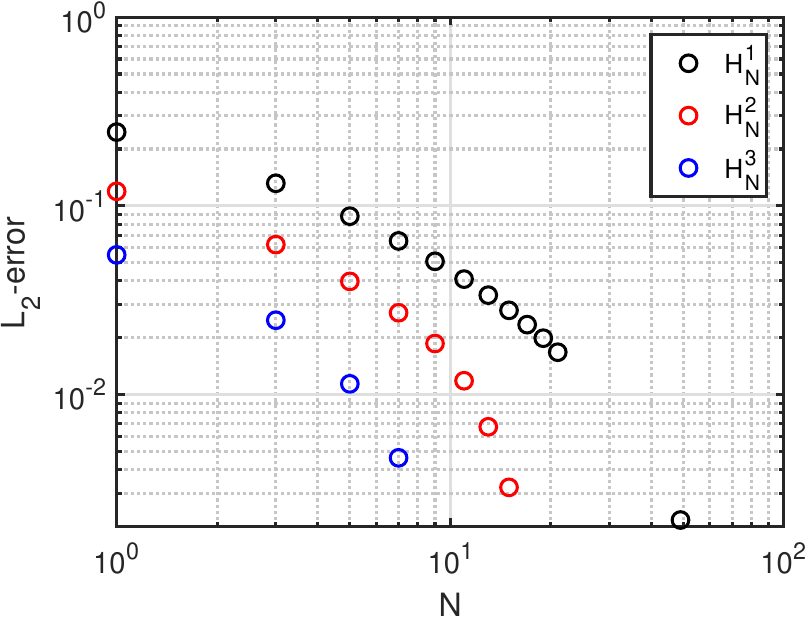}}\hfill
  \subcaptionbox{H$^{T}_{1}$ vs H$^{T}_{3}$ vs H$^{T}_{5}$\label{fig:vacuum_conv_b}}{\includegraphics[width=57mm]{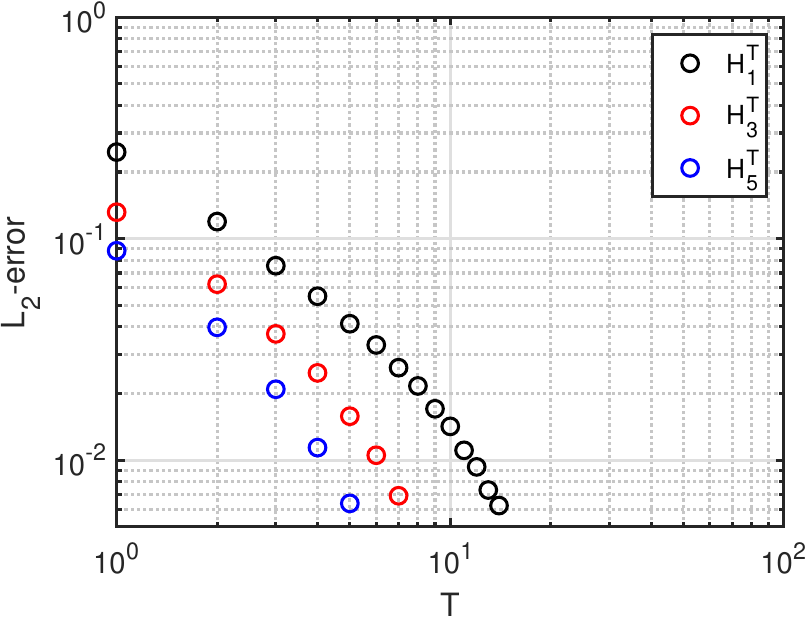}}
  \caption{(\S\ref{sec:vacuum_problem}: Streaming in a vacuum) Convergence study of H$^T_N$ solutions for the vacuum propagation problem at $t=2.5\times10^{-11}$s with $N_z=100$ and $\text{CFL}=0.3$ with respect to (a) $N$ and (b) $T$.}
 \label{fig:HTN_vacuum_conv_rate}
\end{figure}
\begin{figure}[htbp]
 \centering
 \captionsetup[subfigure]{position=top,aboveskip=1pt}
 \subcaptionbox{Vacuum propagation: H$^{1}_{N}$\label{fig:vacuum_a}}{\includegraphics[width=49mm]{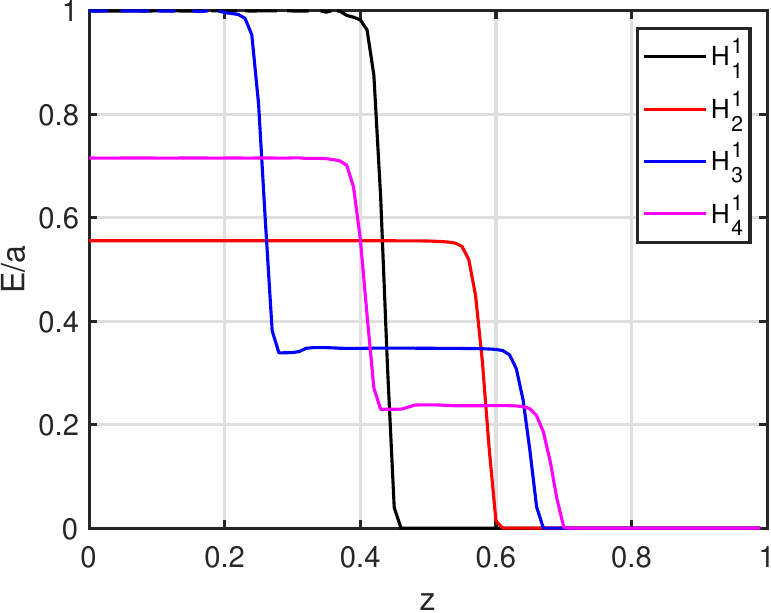}}\hfill
  \subcaptionbox{Vacuum propagation: H$^{2}_{N}$\label{fig:vacuum_b}}{\includegraphics[width=49mm]{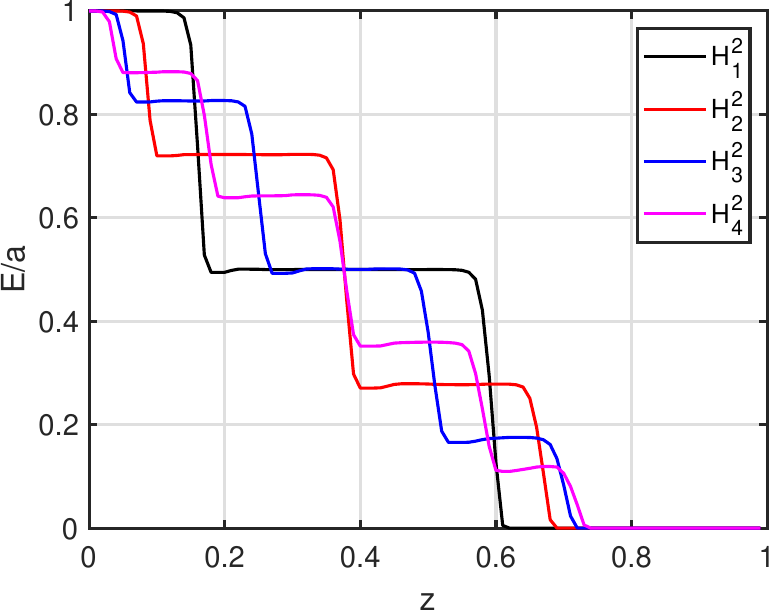}}\vspace{3mm}
   \subcaptionbox{Vacuum propagation: H$^{T}_{1}$\label{fig:vacuum_c}}{\includegraphics[width=49mm]{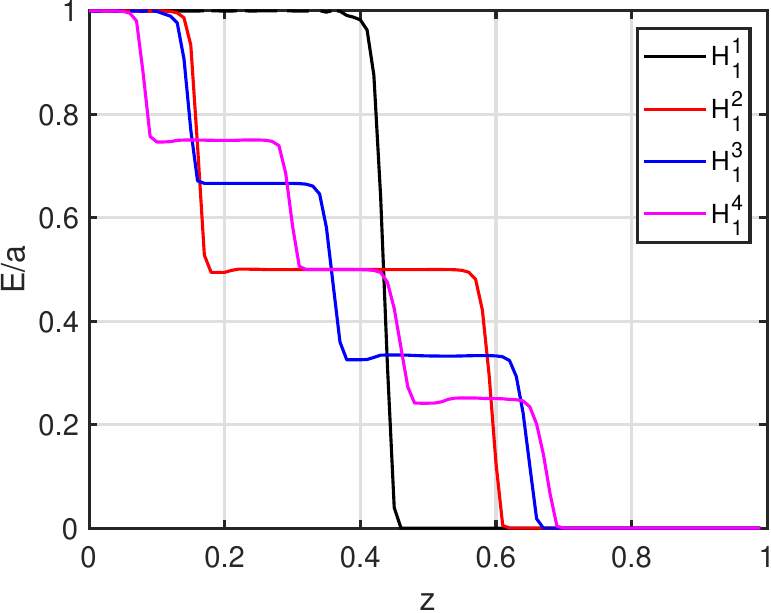}}\hfill
 \subcaptionbox{Vacuum propagation: H$^{T}_{2}$\label{fig:vacuum_d}}{\includegraphics[width=49mm]{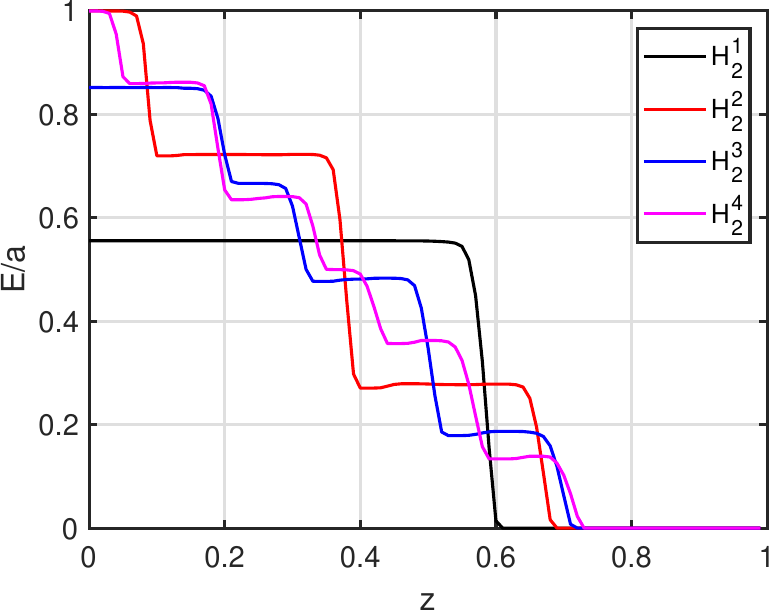}}\vspace{3mm}
 \subcaptionbox{Vacuum propagation: H$^{1}_{15}$\label{fig:vacuum_e}}{\includegraphics[width=49mm]{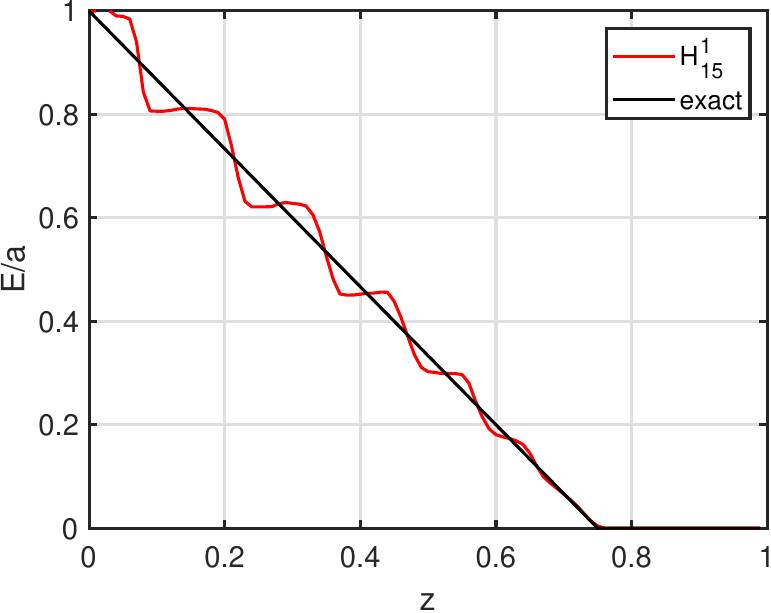}}\hfill
 \subcaptionbox{Vacuum propagation: H$^{2}_{7}$\label{fig:vacuum_f}}{\includegraphics[width=49mm]{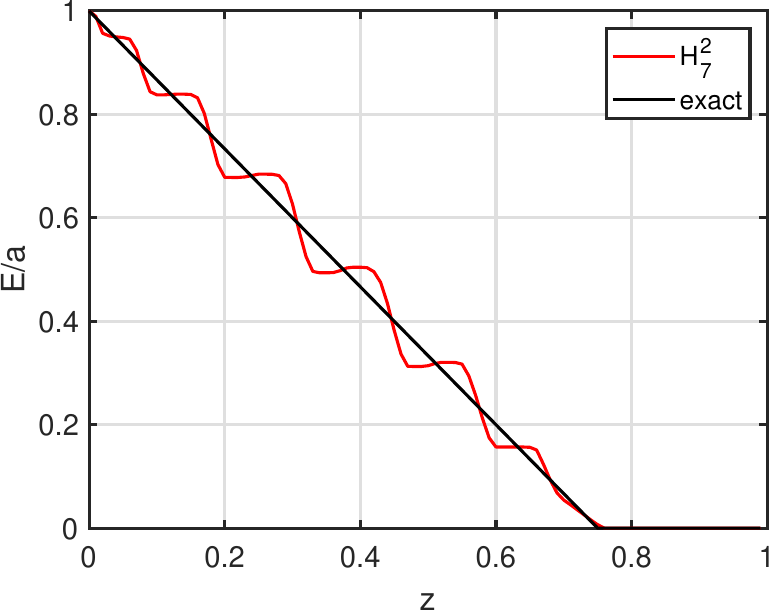}}\vspace{3mm}
  \subcaptionbox{Vacuum propagation: H$^{4}_{3}$\label{fig:vacuum_g}}{\includegraphics[width=49mm]{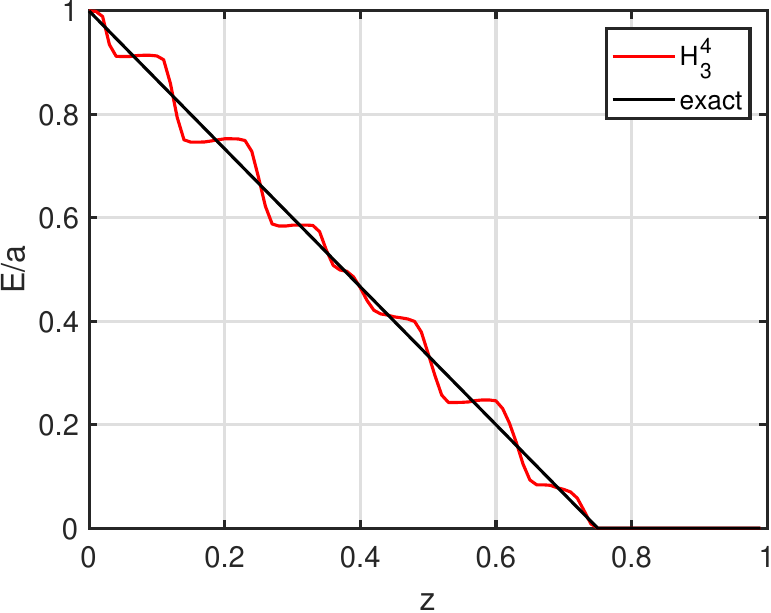}}\hfill
   \subcaptionbox{Vacuum propagation: H$^{8}_{1}$\label{fig:vacuum_h}}{\includegraphics[width=49mm]{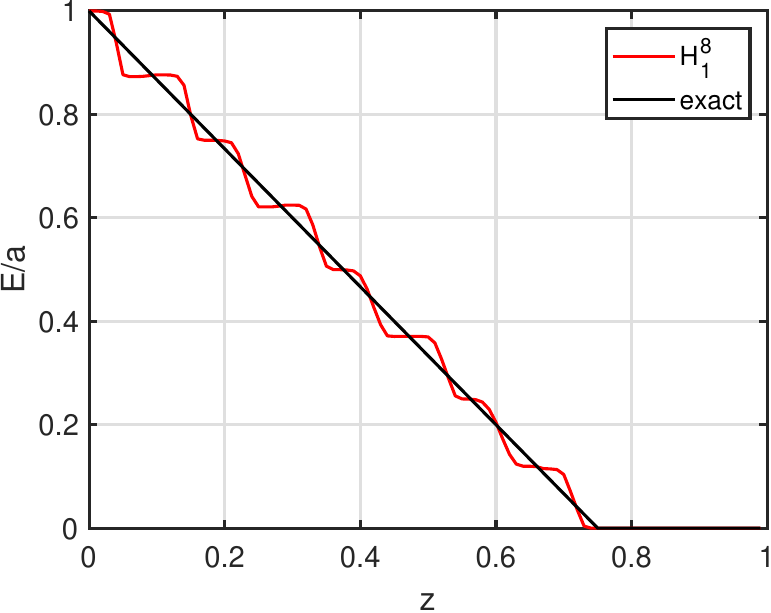}}
 \caption{(\S\ref{sec:vacuum_problem}: Streaming in a vacuum) The H$^T_N$ solutions for the vacuum propagation problem at $t=2.5\times10^{-11}$ with $N_z=100$ and $\text{CFL}=0.3$.}
 \label{fig:HTN_vacuum}
\end{figure}
\begin{table}[htbp]
    \centering
    \begin{tabular}{|c|c|c|c|c|}
    \midrule
    error&H$^8_1$&H$^4_3$&H$^2_7$& $\text{H}^1_{15} \left( =\text{P}_{15} \right)$ \\
    \midrule[1pt]
       $L_{2}$&2.1623$\times10^{-2}$ & 2.4764$\times10^{-2}$& 2.7112$\times10^{-2}$&   2.7934$\times10^{-2}$   \\
       \midrule
       $L_{\infty}$&5.6687$\times10^{-2}$& 6.2532$\times10^{-2}$ & 6.3653$\times10^{-2}$ &   7.3359$\times10^{-2}$\\
       \midrule
    \end{tabular}
    \caption{(\S\ref{sec:vacuum_problem}: Streaming in a vacuum) $L_{2}$ and $L_{\infty}$ errors for the vacuum propagation problem for various methods that all have the same degrees of freedom ($\text{DOF}=16$).}
    \label{tab:error_table_dof_16}
\end{table}
\subsection{\bf Su-Olson problem}\label{sec:su_olson_problem}
The next benchmark problem we solve is the non-equilibrium Su-Olson problem with  material coupling \cite{su-olson}. In order to compare the numerical solutions to the semi-analytic solution, we follow the conditions used in \cite{su-olson}, i.e., $\sigma=1$, and the external source term is given by
\begin{equation}
    S(t,z)= 
\begin{cases}
    ac & \text{if  } \quad -0.5\leq{z}\leq 0.5, \\
    0              & \text{otherwise.}
\end{cases}
\end{equation}
We use the computational domain $z \in \bigl[-ct_\text{end}-1, \, ct_\text{end}+1 \bigr]$ with periodic boundary conditions. For this problem, both $\text{P}_N = \text{H}^1_N$ and H$^T_N$ perform very well even with small $N$. 

\figref{fig:HTN_su_olson} shows the results from three different simulations with $ct_{\text{end}}$=1.0, 3.16, and 10.0, respectively. The solutions to these three cases are shown on a linear scale in \figref{fig:su_olson_a} (P$_3$ and H$_1^2$) and \ref{fig:su_olson_b} (P$_5$ and H$_2^2$),
and on a log-log scale in \figref{fig:su_olson_c} (P$_3$ and H$_1^2$) and \ref{fig:su_olson_d} (P$_5$ and H$_2^2$). The reference solutions are obtained from \cite{su-olson}.

\bigskip

\begin{figure}[htbp]
 \centering
 \captionsetup[subfigure]{position=top,aboveskip=1pt}
 \subcaptionbox{Su-Olson: P$_3$ vs H$^{2}_{1}$\label{fig:su_olson_a}}{\includegraphics[width=56mm]{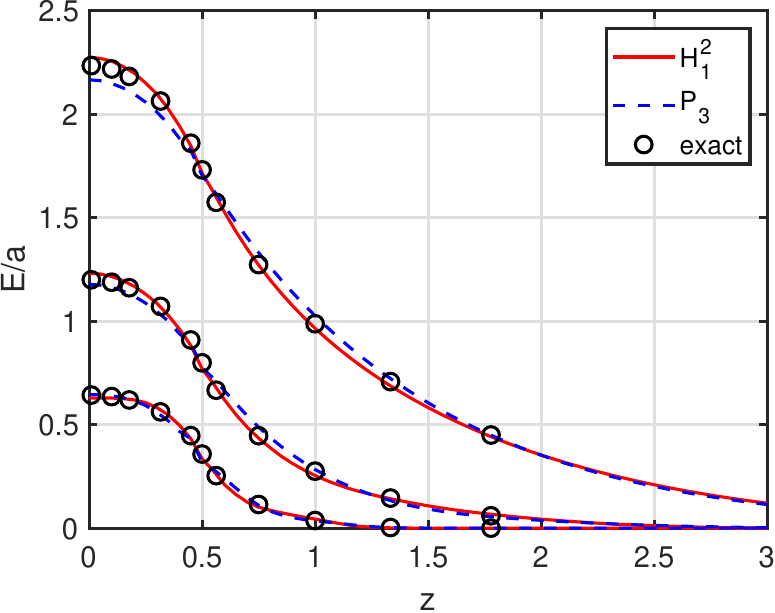}} \hspace{3mm}
  \subcaptionbox{Su-Olson: P$_5$ vs H$^{2}_{2}$\label{fig:su_olson_b}}{\includegraphics[width=56mm]{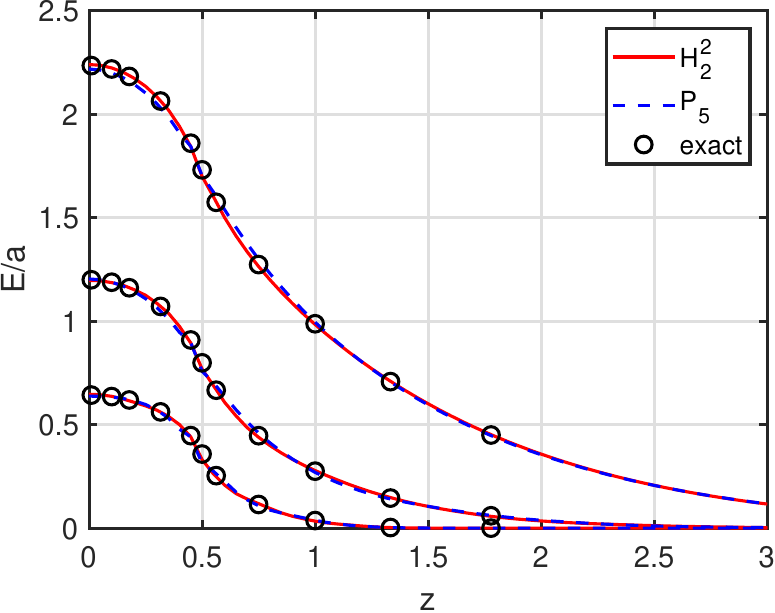}}\vspace{3mm}
 \subcaptionbox{Su-Olson: P$_3$ vs H$^{2}_{1}$ (log-log)\label{fig:su_olson_c}}{\includegraphics[width=56mm]{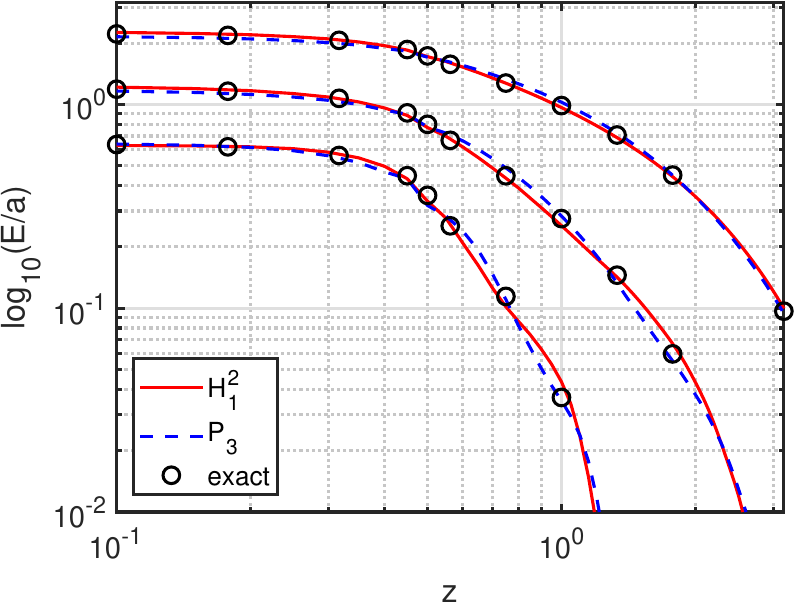}} \hspace{3mm}
\subcaptionbox{Su-Olson: P$_5$ vs H$^{2}_{2}$ (log-log)\label{fig:su_olson_d}}{\includegraphics[width=56mm]{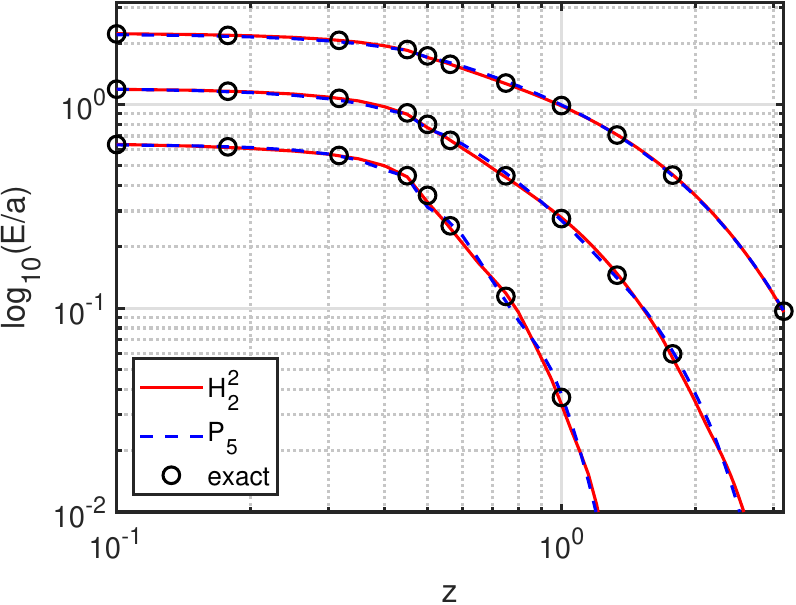}}
 \caption{(\S\ref{sec:su_olson_problem}: Su-Olson problem) Three different simulations using H$^T_N$ for the Su-Olson problem with $N_z=100$, $\text{CFL}=0.3$, and $ct_\text{end}=1.0,\, 3.16,\, 10$, respectively. Panels (a) and (b) show P$_3$, H$_1^2$,
 P$_5$, and H$_2^2$ on a linear scale, while panels (c) and (d) show those same solutions on a log-log scale.}
 \label{fig:HTN_su_olson}
\end{figure}

\subsection{\bf Diffusive Marshak-wave problem}\label{sec:diffusive_Marshak_problem}
In this section, a diffusive Marshak-wave problem is investigated; this problem is a standard test
case in the literature \cite{evans,li,mcclarren1,vikas}. This problem consists of a semi-infinite medium of material with the opacity
\begin{equation}
\sigma=300/{\theta}^3.
\end{equation}
The computational domain is $z \in [0.0, 0.6]$ and the initial conditions are given by the following with ${\theta}_0=10^{-4}$ keV:
\begin{equation}
I\left(t=0,z,\mu\right) = \frac{1}{2}ac {{\theta}_0}^4 \qquad \text{and} \qquad
{\theta}\left(t=0,z\right)={\theta}_0.
\end{equation}

We use the isotropic incoming boundary condition corresponding to a 1 keV temperature source on the left boundary $z_L=0$, and no incoming radiation on the right boundary at $z_R=0.6$:
\begin{equation}
    I\left(t,z_L,\mu>0\right) = \frac{1}{2}ac \qquad \text{and} \qquad
    I\left(t,z_R,\mu<0\right) = 0.
\end{equation}
We compute the material temperature, $\theta(t,z)$ at various times: $t=10^{-8}$s,\ $5\times10^{-8}$s,\ and $10^{-7}$s. In our numerical tests for this problem all H$^T_N$ solutions look  similar, thus, we only present H$^2_2$ solutions in \figref{fig:HTN_diffusive_marshak}. In this test, we use $N_z=16$ with the mesh size $\Delta{z}=0.0375$. Despite the fact that we use a coarse mesh, the H$^T_N$ solution is able to adequately capture the wave propagation front.  Due to the fact that  diffusion dominates, we are able to achieve stable results with a relaxed Courant number: $\text{CFL}=1.7$. The reference solution shown in this plot is the semi-analytic equilibrium-diffusion solution (e.g., see \cite{mcclarren1}).


\begin{figure}[htbp]
 \centering
\includegraphics[width=70mm]{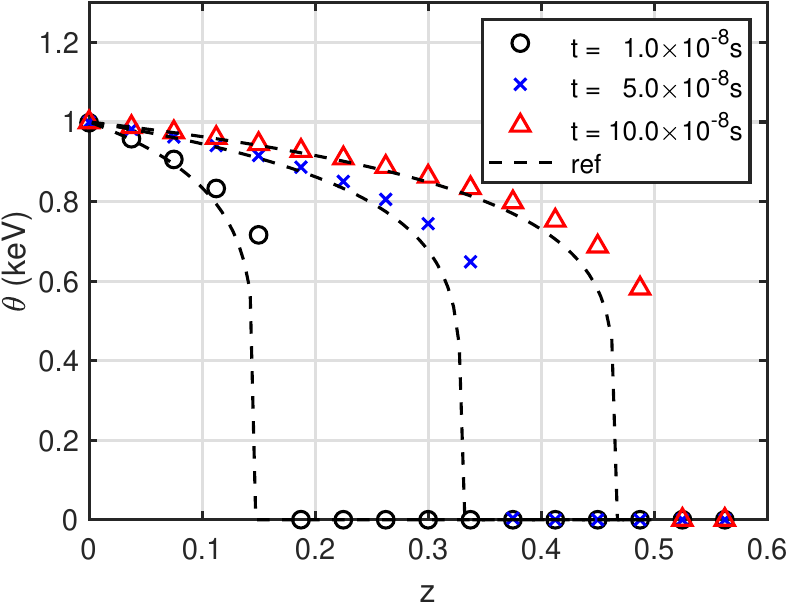}
 \caption{(\S\ref{sec:diffusive_Marshak_problem}: Diffusive Marshak-wave problem) The H$^2_2$ solutions for the material temperature, $\theta(t,z)$, for the diffusive Marshak-wave problem at various times: $t=10^{-8}$s, $5\times10^{-8}$s, and $10^{-7}$s, respectively, with $N_z=16$, $\sigma=300/{\theta}^3$, ${\theta}_0=10^{-4}$ keV, and $\text{CFL}=1.7$. The reference solution is the semi-analytic 
 equilibrium-diffusion solution.}
 \label{fig:HTN_diffusive_marshak}
 \end{figure}

\subsection{\bf Marshak-wave in thin medium}\label{sec:thin_Marshak_problem}
Here we apply our scheme to a Marshak-wave problem in an optically thin medium with an opacity given by 
\begin{equation}
\label{eqn:thin_opacity}
\sigma = 3/\theta^3.
\end{equation}
The radiation temperature is given by 
\begin{equation}
\label{eqn:marshak_theta_rad}
{\theta}_{\text{rad}}(t,z):=\sqrt[4]{\frac{E(t,z)}{a}},
\end{equation}
and the computational domain is $z\in[0,0.35]$.
We use the isotropic incoming boundary condition corresponding to a 1 keV temperature source on the left boundary $z_L=0$, and no incoming radiation on the right boundary at $z_R=0.35$:
\begin{equation}
    I\left(t,z_L,\mu>0\right) = \frac{1}{2}ac \qquad \text{and} \qquad
    I\left(t,z_R,\mu<0\right) = 0.
\end{equation}
The initial conditions are given by the following with ${\theta}_0=10^{-5}$ keV:
\begin{equation}
I\left(t=0,z,\mu\right) = \frac{1}{2}ac {{\theta}_0}^4 \qquad \text{and} \qquad
{\theta}\left(t=0,z\right)={\theta}_0.
\end{equation}

In \figref{fig:HTN_thin_marshak_a} we show the H$^{2}_{2}$ solution with $N_z=400$ for the material temperature,
$\theta$, and the radiation temperature, $\theta_{\text{rad}}$ \eqref{eqn:marshak_theta_rad}. 
In \figref{fig:HTN_thin_marshak_b} we show the material temperature error for H$^{2}_{2}$ for various $N_z$. The error is computed by comparing the H$^2_2$ solutions with various $N_z$ to
 the P$_5$ solution with $N_z=2048$. This figure shows the expected degradation of the order of accuracy to first order in space for this problem due to the discontinuity of the solution.

Furthermore, in  \figref{fig:HTN_thin_marshak_convergence_rate} we investigate the various
convergence rates of H$^{1}_{N}, $H$^{2}_{N}$, H$^{4}_{N}$ and H$^{8}_{N}$ with $N_z=512$ as a function of $N$. In particular, we compute the following approximate $L_2$ error:
\begin{equation}
L_2\text{  error} \, := \, \left[\frac{0.35}{N_z}\sum_{i=1}^{N_z}\Bigl( \theta^{T}_{2j-1}\left(t_{\text{end}}, z_i\right)
- \theta^{T}_{99}\left(t_{\text{end}}, z_i\right) \Bigr)^{2}\right]^{1/2},
\end{equation}
for $j=1,2,\ldots,9$, where $\theta^{T}_N$ represents the material temperature as calculated
with the H$^T_N$ model. 
Due to the non-smoothness of the solution we note a fairly rapid convergence as a 
function of $T$ with fixed $N$, and slower convergence as a function of $N$ with fixed $T$. 


\begin{figure}[htbp]
 \centering
 \captionsetup[subfigure]{position=top,aboveskip=1pt}
  \subcaptionbox{The H$^2_2$ solutions with $N_z=400$\label{fig:HTN_thin_marshak_a}}{\includegraphics[width=57mm]{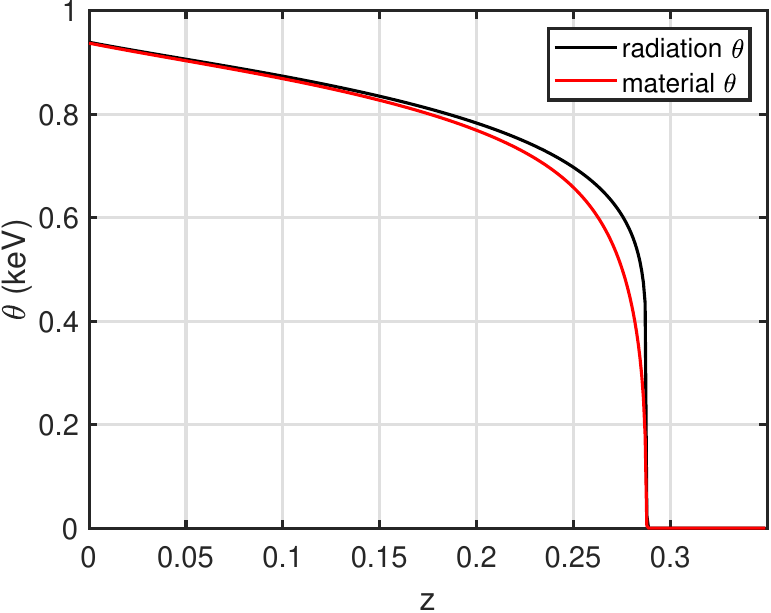}}\hfill
 \subcaptionbox{H$^{2}_{2}$ convergence rate in position space\label{fig:HTN_thin_marshak_b}}{\includegraphics[width=57mm]{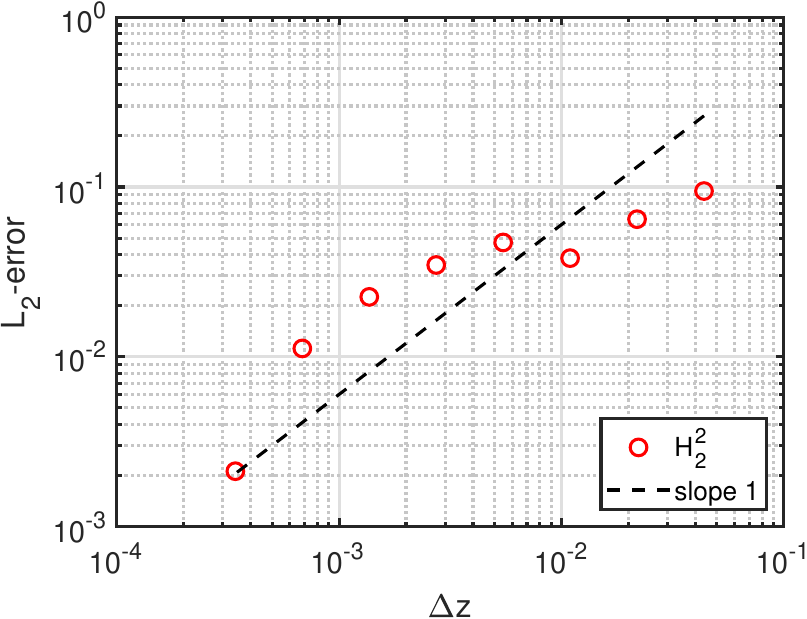}}
 \caption{(\S\ref{sec:thin_Marshak_problem}: Marshak-wave in thin medium) The H$^2_2$ solutions and convergence rate for the thin Marshak-wave problem at $t=10^{-9}$s with $\sigma=3/{\theta}^3$, ${\theta}_0=10^{-5}$ keV, and $\text{CFL}=0.3$.}
 \label{fig:HTN_thin_marshak}
 \end{figure}
 
 
\begin{figure}[htbp]
 \centering
\includegraphics[width=90mm]{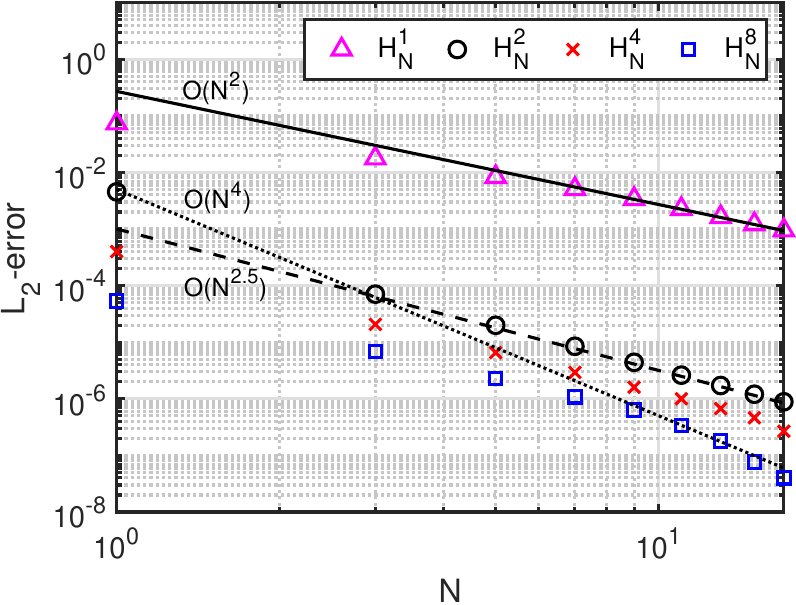}
 \caption{(\S\ref{sec:thin_Marshak_problem}: Marshak-wave in thin medium) Error in various H$^T_N$ approximations with $N_z=512$ and $\text{CFL}=0.3$.}
 \label{fig:HTN_thin_marshak_convergence_rate}
 \end{figure}

\subsection{\bf Smooth Marshak-wave}\label{sec:smooth_Marshak_problem}
Finally, a smooth Marshak-wave problem is considered in this section to observe the convergence rate of our numerical scheme on a smooth solution in an optically thin medium. The opacity
and radiation temperatures are again given by \eqref{eqn:thin_opacity} and
\eqref{eqn:marshak_theta_rad}, respectively.
Following \cite{lowrie,mcclarren1}, the computational domain is $z \in [0,0.8]$ and the smooth initial conditions are given by:
\begin{align}
    I(t=0,z,\mu) &= \frac{ac}{2} \biggl[ 1-0.498{\Bigl(1+\tanh{\bigl[50(z-0.25)\bigr]}\Bigr)} \biggr], \\
    {\theta}(t=0,z)&=\left(\frac{E(t=0,z)}{a}\right)^{\frac{1}{4}}.
\end{align}
The boundary conditions are given by 
\begin{equation}
    I\left(t,z_L,\mu>0\right) = I(t=0,z_L,\mu) \qquad \text{and} \qquad
    I\left(t,z_R,\mu<0\right) = 0,
\end{equation}
where $z_L=0$ and $z_R=0.8$.

In \figref{fig:IC_smooth_Marshak} we show the solution for the material
temperature, $\theta(t,z)$, at various times: (a)  

The initial condition for the material temperature, $\theta(t,z)$, is shown in Panel (a) of \figref{fig:IC_smooth_Marshak}. In  \figref{fig:IC_smooth_Marshak}(b) we show a direct
comparison of the material temperature as computed with $P_5$ and $H^2_2$, both methods have $\text{DOF}=6$, for $N_z = 128$. In  \figref{fig:IC_smooth_Marshak}(c)  and  \figref{fig:IC_smooth_Marshak}(d) we show the material temperature as computed by $P_5$ and $H^2_2$ with various $N_z$, respectively.

\begin{figure}[htbp]
    \centering
    \captionsetup[subfigure]{position=top,aboveskip=1pt}
\subcaptionbox{The initial condition}{\includegraphics[width=57mm]{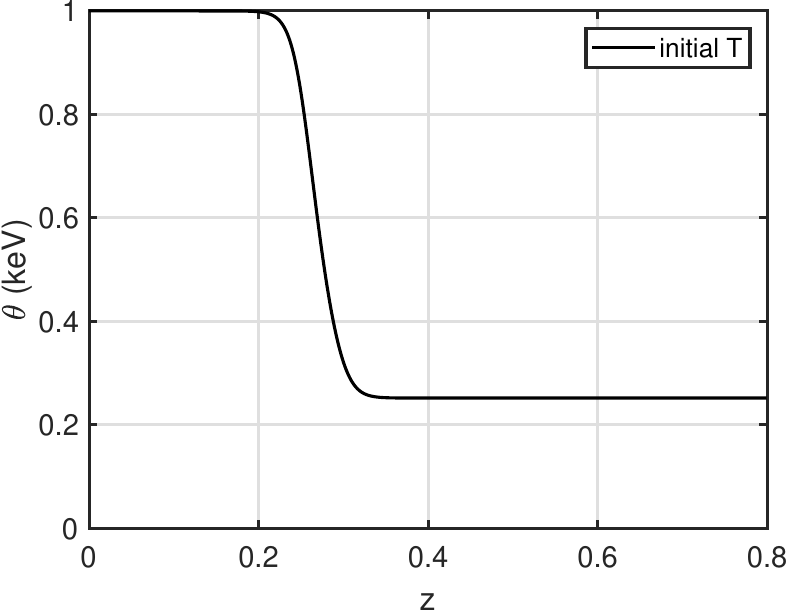}}\hfill
\subcaptionbox{P$_5$ vs H$^{2}_{2}$ with $N_z=128$}{\includegraphics[width=57mm]{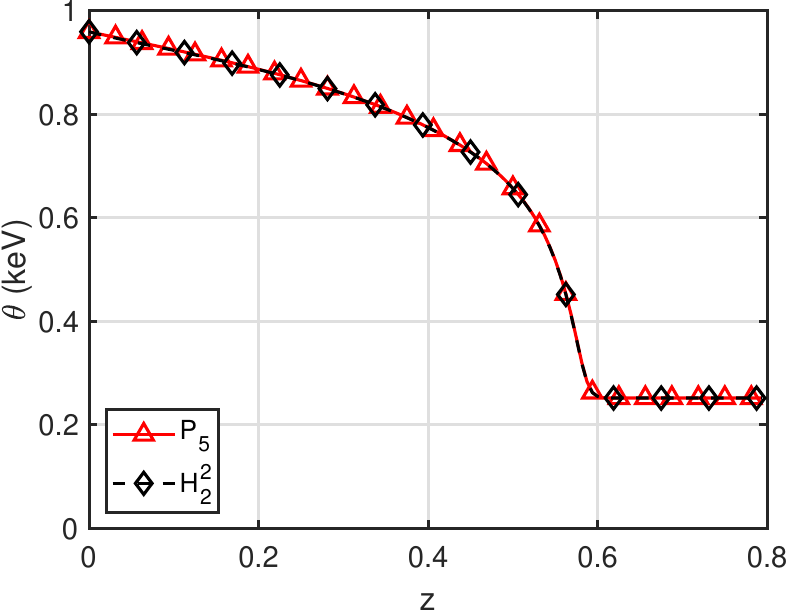}}\vspace{5mm}
 \subcaptionbox{P$_{5}$ solutions with various $N_{z}$}{\includegraphics[width=57mm]{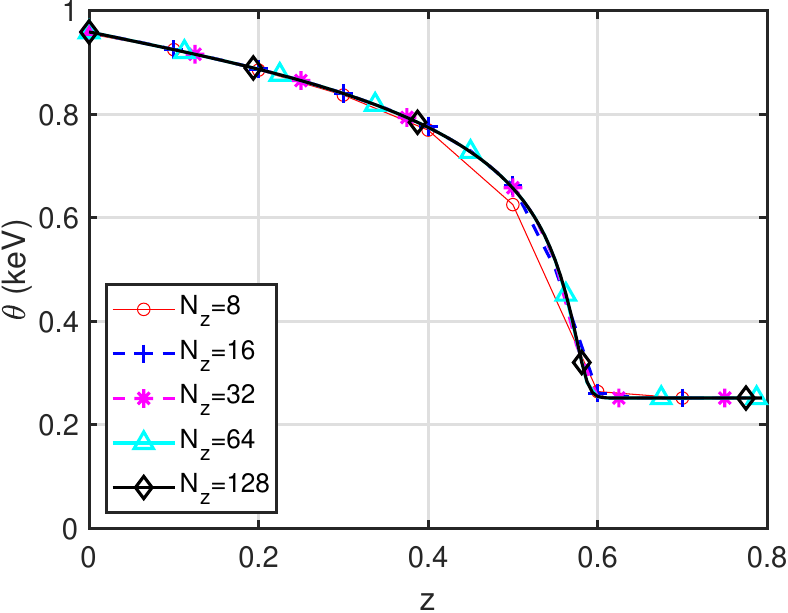}}\hfill
  \subcaptionbox{H$^{2}_{2}$ solutions with various $N_{z}$}{\includegraphics[width=57mm]{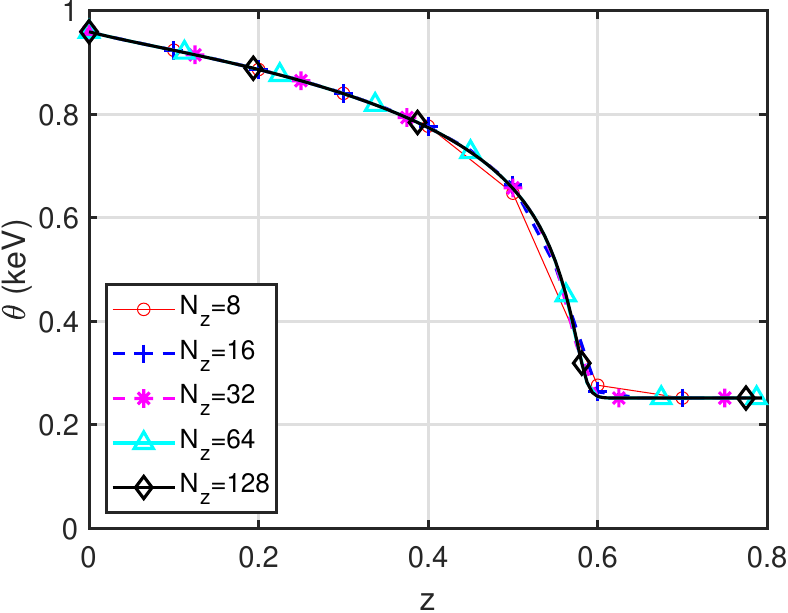}}
    \caption{(\S\ref{sec:smooth_Marshak_problem}: Smooth Marshak-wave) Comparison of the P$_5$ and H$^2_N$ solutions for the material temperature for the smooth Marshak wave problem with $\text{CFL}=0.3$.}
    \label{fig:IC_smooth_Marshak}
\end{figure}

Finally, in \figref{fig:convergence_rate_smooth} we demonstrate second-order convergence 
of the $H^2_2$ with increasing mesh resolution $N_z$. In this figure, 
 the dashed-line indicates a slope of two on a log-log scale. In these convergence experiments the approximate $L_2$ error is calculated via the formula:
\begin{equation}
\label{eqn:l2smooth_err1}
L_2({\theta}-{{\theta}}^{2048}):=\sqrt{\frac{0.8}{N_z}\sum_{k=1}^{N_z}{\left({\theta}_k-{{\theta}}^{2048/(2^m)}_k\right)^2}},
\end{equation}
where
\begin{equation}
\label{eqn:l2smooth_err2}
{{\theta}}^{2048/(2^{m+1})}_k=\frac{1}{2}\left({{\theta}}_{2(k-1)+1}^{2048/(2^{m})}+{{\theta}}_{2(k-1)+2}^{2048/(2^{m})}\right),
\end{equation}
for $m=0,1,\cdots,7$. Here ${{\theta}}^{2048}$ represents the reference temperature solution on mesh
with $N_z=2048$ cells, and the superscript $k$ stands for the $k^{\text{th}}$ grid cell. 
The idea encapsulated in formulas \eqref{eqn:l2smooth_err1} and \eqref{eqn:l2smooth_err2} is that we project the reference solution ${\theta}^{2048}$ onto coarser mesh, i.e., $N_z=2048/(2^m)$, for $m=1,2,\cdots,8$ by taking the average of left and right cell on finer mesh, to obtain ${\theta}^{1024},{\theta}^{512},\cdots,{\theta}^{8}$; once we have projected this solution down to the mesh on which $\theta$ is defined, we can directly compute the $L_2$ distance.
In \figref{fig:convergence_rate_smooth} we have used this strategy with $N_z=8$, 16, 32, 64, 128, 256, 512, and 1024 to show the convergence rate. 

\begin{figure}[htbp]
    \centering
\includegraphics[width=60mm]{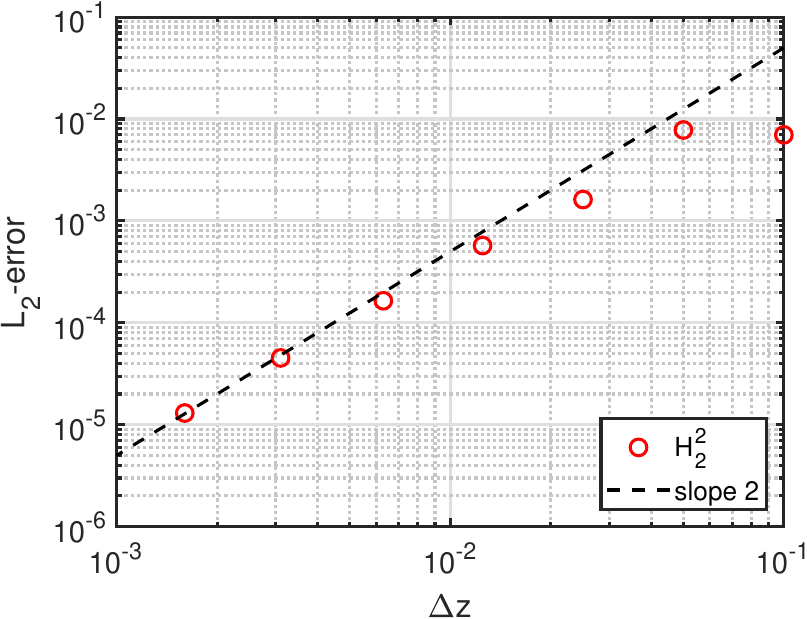}
    \caption{(\S\ref{sec:smooth_Marshak_problem}: Smooth Marshak-wave) Convergence rate of the material temperature ${\theta}$ for the smooth Marshak wave problem.}
    \label{fig:convergence_rate_smooth}
\end{figure}


\section{Conclusion}\label{sec:conclusion}
In this work we have developed the hybrid discrete (H$^T_N$) approximation method for the thermal radiative transfer (TRT) equations, and implemented a numerical discretization of these equations
using a second order discontinuous Galerkin finite element method in conjunction
with a semi-implicit time-stepping scheme. The H$^T_N$ approach acquires desirable properties of two classical methods for the TRT equations, namely P$_N$ (spherical harmonics) and S$_N$ (discrete ordinates), and indeed reduces to each of these approximations in various limits: H$^1_N$ $\equiv$ P$_N$ and H$^T_0$ $\equiv$ S$_T$. We proved that the H$^T_N$ approximation results in a system of hyperbolic partial differential equations for all $T\ge 1$ and $N\ge 0$. In particular, in one spatial dimension, the H$^T_N$ scheme is essentially a collection of P$_N$ approximations localized in each velocity band. Because of this structure, H$^T_N$, just like P$_N$, can exhibit negative
densities. However, because H$^T_N$ has band-localized structure, we  are able to control unphysical numerical oscillations by increasing the number of discrete regions in velocity space, thereby
blending properties of P$_N$ and S$_N$. 

Once the H$^T_N$ approximation was developed for TRT, we introduced a semi-implicit numerical method that is based on a second order explicit Runge-Kutta scheme for the the streaming term and an implicit Euler scheme for the material coupling term. Furthermore, in order to solve the material energy equation implicitly after each predictor and corrector step, we linearized the temperature term using a Taylor expansion; this avoided the need for an iterative procedure, and therefore improved efficiency. In order to reduce unphysical oscillation,  we applied a slope limiter after each time step.

In the numerical results section we compared the solutions of the H$^T_N$ and P$_N$ schemes for the various benchmark problems. We demonstrated for a variety of problems that for a fixed total number of moments, we are able to achieve accuracy gains over the $\text{P}_N \equiv \text{H}^1_N$ approximation by balancing $T$ and $N$. With a more balanced choice of $T$ and $N$, H$^T_N$ shows less oscillation than P$_N$, especially in the presence of discontinuities. {\color{black}One may use the large $T$ when the problems include strong shocks or the solutions are non-smooth. Otherwise, one can use high $N$ instead. The control of the variables $T$ and $N$ can compensate for the shortcomings of P$_N$ and S$_N$ schemes.}
 
In future work, we will develop extensions of the H$_N^T$ method for multi-energy group models or frequency-dependent equations in multiple dimensions. In particular, just as we have done in this work, we will investigate various choices of $N$ and $T$ to achieve accurate and efficient moment closure in the multidimensional setting. We will also investigate adaptive strategies for selecting $T$ and $N$ in the presence of some appropriate error indicator.

\begin{acknowledgements}
JAR was supported in part by NSF Grants DMS--1620128 and DMS--2012699.
\end{acknowledgements}

%
%

\bibliographystyle{spmpsci}      

\begin{thebibliography}{10}
\providecommand{\url}[1]{{#1}}
\providecommand{\urlprefix}{URL }
\expandafter\ifx\csname urlstyle\endcsname\relax
  \providecommand{\doi}[1]{DOI~\discretionary{}{}{}#1}\else
  \providecommand{\doi}{DOI~\discretionary{}{}{}\begingroup
  \urlstyle{rm}\Url}\fi

\bibitem{adams2020provably}
Adams, M.P., Adams, M.L., Hawkins, W.D., Smith, T., Rauchwerger, L., Amato,
  N.M., Bailey, T.S., Falgout, R.D., Kunen, A., Brown, P.: Provably optimal
  parallel transport sweeps on semi-structured grids.
\newblock J. Comput. Phys. \textbf{407}, 109--234 (2020)

\bibitem{brunner1}
Brunner, T.A., Holloway, J.P.: Two-dimensional time-dependent {R}iemann solvers
  for neutron transport.
\newblock J. Comput. Phys. \textbf{210}(1), 386--399 (2005)

\bibitem{carlson}
Carlson, B.: {Solution of the transport equation by the S$_N$ method}.
\newblock Los Alamos National Laboratory  (1955)

\bibitem{carlson1}
Carlson, B.: {Tables of symmetric equal weight quadrature EQn over the unit
  sphere}.
\newblock Los Alamos National Laboratory  (1971)

\bibitem{article:Cohen1996}
Cohen, A.: An algebraic approach to certain differential eigenvalue problems.
\newblock Linear Algebra Appl. \textbf{240}, 183--198 (1996)

\bibitem{dubroca1999etude}
Dubroca, B., Feugeas, J.L.: Etude th{\'e}orique et num{\'e}rique d'une
  hi{\'e}rarchie de mod{\`e}les aux moments pour le transfert radiatif.
\newblock Comptes Rendus de l'Acad{\'e}mie des Sciences-Series I-Mathematics
  \textbf{329}(10), 915--920 (1999)

\bibitem{evans}
Evans, T.M., Urbatsch, T.J., Lichtenstein, H., Morel, J.E.: {A residual Monte
  Carlo method for discrete thermal radiative diffusion}.
\newblock J. Comput. Phys. \textbf{189}(2), 539--556 (2003)

\bibitem{fan}
Fan, Y.W., Li, R., Zheng, L.C.: {A Nonlinear Moment Model for Radiative
  Transfer Equation in Slab Geometry}.
\newblock J. Comput. Phys. \textbf{404}, 109--128 (2020)

\bibitem{fleck}
Fleck, J.A., Jr., Cummings, J.D.: {An implicit Monte Carlo scheme for
  calculating time and frequency dependent nonlinear radiation transport}.
\newblock J. Comput. Phys. \textbf{8}(3), 313--342 (1971)

\bibitem{gustaffson}
Gustafsson, B., Kreiss, H.O., Oliger, J.: {Time-Dependent Problems and
  Difference Methods, 2nd Edition}.
\newblock Wiley, New York, United States (2013)

\bibitem{hauck2010positive}
Hauck, C., McClarren, R.: Positive p\_n closures.
\newblock SIAM Journal on Scientific Computing \textbf{32}(5), 2603--2626
  (2010)

\bibitem{hauck2011high}
Hauck, C.D.: High-order entropy-based closures for linear transport in slab
  geometry.
\newblock Communications in Mathematical Sciences \textbf{9}(1), 187--205
  (2011)

\bibitem{hauck}
Hauck, C.D., McClarren, R.G.: {Positive P$_N$ closures}.
\newblock SIAM J. Sci. Comput. \textbf{32}(5), 2603--2626 (2010)

\bibitem{lorenz}
J.~Lorenz, H.J.S.: {Stiff well-posedness for hyperbolic systems with large
  relaxation terms (linear constant-coefficient problems)}.
\newblock Adv. Differ. Equ. \textbf{2}(4), 643--666 (1997)

\bibitem{jarrell}
Jarrell, J., Adams, M.: Discrete-ordinates quadrature sets based on linear
  discontinuous finite elements.
\newblock In: Proc. International Conference on Mathematics and Computational
  Methods applied to Nuclear Science and Engineering, Rio de Janeiro, Brazil
  (2011)

\bibitem{klar}
Klar, A.: {An asymptotic-induced scheme for nonstationary transport equations
  in the diffusive limit}.
\newblock SIAM J. Numer. Anal. \textbf{35}(3), 1073--1094 (1998)

\bibitem{klar1}
Klar, A., Unterreiter, A.: {Uniform stability of a finite difference scheme for
  transport equations in diffusive regimes}.
\newblock SIAM J. Numer. Anal. \textbf{40}(3), 891--913 (2002)

\bibitem{koch}
Koch, R., Krebs, W., Wittig, S., Viskanta, R.: {The discrete ordinate
  quadrature schemes for multidimensional radiative transfer}.
\newblock J. Quant. Spectrosc. Ra. \textbf{53}(4), 353--372 (1995)

\bibitem{laiu}
Laiu, M.P., Hauck, C.D., McClarren, R.G., O'Leary, D.P., Tits, A.L.: {Positive
  filtered P$_N$ moment closures for linear kinetic equations}.
\newblock SIAM J. Numer. Anal. \textbf{54}(6), 3214--3238 (2016)

\bibitem{article:Larsen1983}
Larsen, E., Pomraning, G., Badham, V.: Asymptotic analysis of radiative
  transfer equations.
\newblock J. Quant. Spectros. Radiat. Transf. \textbf{29}(4), 285--310 (1983)

\bibitem{lathrop3}
Lathrop, K.D.: Ray effects in discrete ordinates equations.
\newblock Nucl. Sci. Eng. \textbf{32}(3), 357--369 (1968)

\bibitem{lathrop2}
Lathrop, K.D.: Remedies for ray effects.
\newblock Nucl. Sci. Eng. \textbf{45}(3), 255--268 (1971)

\bibitem{lathrop1}
Lathrop, K.D., Carlson, B.G.: {Discrete ordinates angular quadrature of the
  neutron transport equation}.
\newblock Los Alamos Scientific Laboratory Report 3186  (1965)

\bibitem{lau}
Lau, C., Adams, M.: {Discrete Ordinates Quadratures Based on Linear and
  Quadratic Discontinuous Finite Elements over Spherical Quadrilaterals}.
\newblock Nucl. Sci. Eng. \textbf{185}(1), 36--52 (2017)

\bibitem{lewis}
Lewis, E.E., Miller, W.F.: Computational Methods of Neutron Transport.
\newblock John Wiley \& Sons, DeKalb (1994)

\bibitem{li}
Li, W., Liu, C., Zhu, Y., Zhang, J., Xu, K.: {Unified gas-kinetic wave-particle
  methods III: Multiscale photon transport}.
\newblock J. Comput. Phys. \textbf{408}, 109--280 (2020)

\bibitem{lowrie}
Lowrie, R.B.: {A Comparison of Implicit Time Integration Methods for Nonlinear
  Relaxation and Diffusion}.
\newblock J. Comput. Phys. \textbf{196}(2), 566--590 (2004)

\bibitem{mcclarren1}
McClarren, R.G., Evans, T.M., Lowrie, R.B., Densmore, J.D.: {Semi-implicit time
  integration for P$_N$ thermal radiative transfer}.
\newblock J. Comput. Phys. \textbf{227}(16), 7561--7586 (2008)

\bibitem{mcclarren2008solutions}
McClarren, R.G., Holloway, J.P., Brunner, T.A.: On solutions to the pn
  equations for thermal radiative transfer.
\newblock Journal of Computational Physics \textbf{227}(5), 2864--2885 (2008)

\bibitem{mcclarren2}
Mcclarren, R.G., Lowrie, R.B.: {The effects of slope limiting on
  asymptotic-preserving numerical methods for hyperbolic conservation laws}.
\newblock J. Comput. Phys. \textbf{227}(23), 9711--9726 (2008)

\bibitem{mcclarren2009modified}
McClarren, R.G., Urbatsch, T.J.: A modified implicit monte carlo method for
  time-dependent radiative transfer with adaptive material coupling.
\newblock Journal of Computational Physics \textbf{228}(16), 5669--5686 (2009)

\bibitem{olson1}
Olson, G.: {Second-order time evolution of P N equations for radiation
  transport}.
\newblock J. Comput. Phys. \textbf{228}(8), 3072--3083 (2009)

\bibitem{olson}
Olson, G.L., Auer, L.H., Hall, M.L.: Diffusion, p$_1$, and other approximate
  forms of radiation transport.
\newblock J. Quant. Spectrosc. Ra. \textbf{64}(6), 619--634 (2000)

\bibitem{parlett}
Parlett, B.N.: {The Symmetric Eigenvalue Problem}.
\newblock Prentice‐Hall Series in Comput. Math. \textbf{61}(7), 277--348
  (1981)

\bibitem{pomraning}
Pomraning, G.C.: Variational boundary conditions for the spherical harmonics
  approximation to the neutron transport equation.
\newblock Ann. Phys. \textbf{27}, 193--215 (1964)

\bibitem{pomraning1}
Pomraning, G.C.: The Equations of Radiation Hydrodynamics.
\newblock Pergamon Press, Oxford, U.K. (1973)

\bibitem{shin}
Shin, M.: {Hybrid discrete (H$^T_N$) approximations to the equation of
  radiative transfer}.
\newblock Ph.D. thesis, Iowa State University, Ames, IA (2019)

\bibitem{siegel}
Siegel, R., Howell, J.R.: {Thermal radiation heat transfer}, vol.~3.
\newblock Hemisphere Publishing Corp., Washington, United States (1972)

\bibitem{su-olson}
Su, B., Olson, G.L.: {An analytical benchmark for non-equilibrium radiative
  transfer in an isotropically scattering medium}.
\newblock Ann. Nucl. Energy \textbf{24}(13), 1035--1055 (1997)

\bibitem{thurgood}
Thurgood, C.P., Pollard, A., Becker, H.A.: {The T$_N$ Quadrature Set for the
  Discrete Ordinates Method}.
\newblock J. Heat Transfer \textbf{117}(4), 1068--1070 (1995)

\bibitem{vikas}
Vikas, V., Hauck, C., Wang, Z., Fox, R.: {Radiation transport modeling using
  extended quadrature method of moments}.
\newblock J. Comput. Phys. \textbf{246}(1), 221--241 (2013)

\bibitem{web:wikipedia_GL}
Wikipedia: Gaussian quadrature (2021).
\newblock \urlprefix\url{https://en.wikipedia.org/wiki/Gaussian_quadrature}

\bibitem{wollaber2016four}
Wollaber, A.B.: Four decades of implicit monte carlo.
\newblock Journal of Computational and Theoretical Transport \textbf{45}(1-2),
  1--70 (2016)

\bibitem{zheng2016moment}
Zheng, W., McClarren, R.G.: Moment closures based on minimizing the residual of
  the pn angular expansion in radiation transport.
\newblock Journal of Computational Physics \textbf{314}, 682--699 (2016)

\end{thebibliography}

\end{document}